\documentclass[11pt,letterpaper]{article}

\usepackage{fullpage}

\def\stocmode{0}

\def\arxivmode{0}
\def\fastmode{0}

\def\showauthornotes{2}

\def\showkeys{0}
\def\showdraftbox{1}
\def\showcolorlinks{1}
\def\usemicrotype{1}
\def\showfixme{1}

\newcommand{\1}{\vvmathbb{1}}

\newcommand{\slog}[1]{\sqrt{\vphantom{M_k}\smash[b]{\log #1}}}

\newcommand{\rightspace}{\qquad\qquad\qquad\qquad}

\ifnum\arxivmode=0  \else  \fi

\ifnum\fastmode=0
\usepackage{etex}
\fi

\ifnum\fastmode=0
\usepackage[l2tabu, orthodox]{nag}
\fi

\usepackage{xspace,enumerate}

\usepackage[dvipsnames]{xcolor}

\ifnum\fastmode=0
\usepackage[T1]{fontenc}
\usepackage[full]{textcomp}
\fi

\usepackage[american]{babel}

\usepackage{mathtools}

\usepackage{amsthm}

\newtheorem{theorem}{Theorem}[section]
\newtheorem*{theorem*}{Theorem}

\newtheorem*{proposition*}{Proposition}
\newtheorem{lemma}[theorem]{Lemma}
\newtheorem*{lemma*}{Lemma}
\newtheorem{corollary}[theorem]{Corollary}
\newtheorem*{conjecture*}{Conjecture}

\newtheorem*{fact*}{Fact}

\newtheorem*{exercise*}{Exercise}

\newtheorem*{hypothesis*}{Hypothesis}

\theoremstyle{definition}
\newtheorem{definition}[theorem]{Definition}

\newtheorem{exercise-easy}[theorem]{Exercise}
\newtheorem{exercise-med}[theorem]{Exercise}
\newtheorem{exercise-hard}[theorem]{Exercise$^\star$}

\newtheorem*{claim*}{Claim}

\newtheorem{remark}[theorem]{Remark}
\newtheorem*{remark*}{Remark}

\newtheorem*{observation*}{Observation}

\ifnum\arxivmode=0
\usepackage{newpxtext} %
\usepackage{textcomp} %
\usepackage[varg,bigdelims]{newpxmath}
\usepackage[scr=rsfso]{mathalfa}%
\usepackage{bm} %
\let\mathbb\varmathbb
\fi

\ifnum\arxivmode=1
\usepackage[varg]{pxfonts} %
\usepackage{textcomp} %
\usepackage[scr=rsfso]{mathalfa}%
\usepackage{bm} %
\fi

\ifnum\showkeys=1
\usepackage[color]{showkeys}
\fi

\makeatletter
\@ifpackageloaded{hyperref}
{}
{\usepackage[pagebackref]{hyperref}}

\definecolor{bleudefrance}{rgb}{0.01, 0.1, 1.0}
\definecolor{azure}{rgb}{0.0, 0.5, 1.0}

\ifnum\showcolorlinks=1
\hypersetup{
pagebackref=true,
colorlinks=true,
urlcolor=blue,
linkcolor=blue,
citecolor=OliveGreen}
\fi

\ifnum\showcolorlinks=0
\hypersetup{
pagebackref=true,
colorlinks=false,
pdfborder={0 0 0}}
\fi

\usepackage{prettyref}

\newcommand{\savehyperref}[2]{\texorpdfstring{\hyperref[#1]{#2}}{#2}}

\newrefformat{eq}{\savehyperref{#1}{\textup{(\ref*{#1})}}}
\newrefformat{lem}{\savehyperref{#1}{Lemma~\ref*{#1}}}
\newrefformat{def}{\savehyperref{#1}{Definition~\ref*{#1}}}
\newrefformat{obs}{\savehyperref{#1}{Observation~\ref*{#1}}}
\newrefformat{ass}{\savehyperref{#1}{Assumption~\ref*{#1}}}
\newrefformat{thm}{\savehyperref{#1}{Theorem~\ref*{#1}}}
\newrefformat{cor}{\savehyperref{#1}{Corollary~\ref*{#1}}}
\newrefformat{cha}{\savehyperref{#1}{Chapter~\ref*{#1}}}
\newrefformat{sec}{\savehyperref{#1}{Section~\ref*{#1}}}
\newrefformat{app}{\savehyperref{#1}{Appendix~\ref*{#1}}}
\newrefformat{tab}{\savehyperref{#1}{Table~\ref*{#1}}}
\newrefformat{fig}{\savehyperref{#1}{Figure~\ref*{#1}}}
\newrefformat{hyp}{\savehyperref{#1}{Hypothesis~\ref*{#1}}}
\newrefformat{alg}{\savehyperref{#1}{Algorithm~\ref*{#1}}}
\newrefformat{rem}{\savehyperref{#1}{Remark~\ref*{#1}}}
\newrefformat{item}{\savehyperref{#1}{Item~\ref*{#1}}}
\newrefformat{step}{\savehyperref{#1}{step~\ref*{#1}}}
\newrefformat{conj}{\savehyperref{#1}{Conjecture~\ref*{#1}}}
\newrefformat{fact}{\savehyperref{#1}{Fact~\ref*{#1}}}
\newrefformat{prop}{\savehyperref{#1}{Proposition~\ref*{#1}}}
\newrefformat{prob}{\savehyperref{#1}{Problem~\ref*{#1}}}
\newrefformat{claim}{\savehyperref{#1}{Claim~\ref*{#1}}}
\newrefformat{relax}{\savehyperref{#1}{Relaxation~\ref*{#1}}}
\newrefformat{red}{\savehyperref{#1}{Reduction~\ref*{#1}}}
\newrefformat{part}{\savehyperref{#1}{Part~\ref*{#1}}}
\newrefformat{ex}{\savehyperref{#1}{Exercise~\ref*{#1}}}
\newrefformat{property}{\savehyperref{#1}{Property~\ref*{#1}}}

\newcommand{\Sref}[1]{\hyperref[#1]{\S\ref*{#1}}}

\usepackage{nicefrac}

\ifnum\fastmode=0
\ifnum\usemicrotype=1
\usepackage{microtype}
\fi
\fi

\ifnum\showauthornotes=2
\newcommand{\mynotes}[1]{{\sffamily\small\color{teal}{#1}}\medskip}
\newcommand{\Authornote}[2]{{\sffamily\small\color{Maroon}{[#1: #2]}}\medskip}
\newcommand{\Authornotecolored}[3]{{\sffamily\small\color{#1}{[#2: #3]}}}
\newcommand{\Authorcomment}[2]{{\sffamily\small\color{gray}{[#1: #2]}}}
\newcommand{\Authorstartcomment}[1]{\sffamily\small\color{gray}[#1: }

\newcommand{\Authorfnote}[2]{\footnote{\color{red}{#1: #2}}}
\newcommand{\Authorfixme}[1]{\Authornote{#1}{\textbf{??}}}
\newcommand{\Authormarginmark}[1]{\marginpar{\textcolor{red}{\fbox{\Large #1:!}}}}
\newcommand{\myexplain}[1]{{\sffamily\small\color{red}{\noindent [Explanation:\medskip\newline \begin{quote}#1\hfill]\end{quote}}}\medskip}
\newcommand{\explain}[1]{{\sffamily\small\color{red}{#1}}\medskip}

\else

\newcommand{\mynotes}[1]{}
\newcommand{\Authornote}[2]{}
\newcommand{\Authornotecolored}[3]{}
\newcommand{\Authorcomment}[2]{}
\newcommand{\Authorstartcomment}[1]{}

\newcommand{\Authorfnote}[2]{}
\newcommand{\Authorfixme}[1]{}
\newcommand{\Authormarginmark}[1]{}
\newcommand{\myexplain}[1]{}
\newcommand{\explain}[1]{}

\ifnum\showauthornotes=1
\renewcommand{\myexplain}[1]{{\sffamily\small\color{red}{\noindent \begin{quote}{\bf Explanation:} \medskip\newline #1\end{quote}}}\medskip}
\fi

\fi

\ifnum\showfixme=0

\fi

\usepackage{boxedminipage}

\newcommand{\Esymb}{\mathbb{E}}
\newcommand{\Psymb}{\mathbb{P}}

\DeclareMathOperator*{\E}{\Esymb}

\DeclareMathOperator*{\ProbOp}{\Psymb}

\renewcommand{\Pr}{\ProbOp}

\newcommand{\textparen}[1]{\text{(#1)}}

\ifx\because\undefined
\newcommand{\because}[1]{\textparen{because #1}}
\else
\renewcommand{\because}[1]{\textparen{because #1}}
\fi

\newcommand{\seteq}{\mathrel{\mathop:}=}

\newcommand\bdot\bullet

\ifx\mathds\undefined %

\else

\fi

\DeclareMathOperator{\poly}{poly}

\DeclareMathOperator{\conv}{conv}

\DeclareMathOperator{\rank}{rank}

\newcommand{\R}{\mathbb R}

\newcommand{\cA}{\mathcal A}

\newcommand{\cE}{\mathcal E}

\newcommand{\cN}{\mathcal N}

\newcommand{\cS}{\mathcal S}

\newcommand{\sfK}{\mathsf K}

\newcommand{\sfR}{\mathsf R}

\renewcommand{\leq}{\leqslant}

\renewcommand{\geq}{\geqslant}

\ifnum\showdraftbox=1

\else

\fi

\let\epsilon=\varepsilon

\numberwithin{equation}{section}

\newcommand\MYcurrentlabel{xxx}

\newcommand{\MYstore}[2]{%
  \global\expandafter \def \csname MYMEMORY #1 \endcsname{#2}%
}

\newcommand{\MYload}[1]{%
  \csname MYMEMORY #1 \endcsname%
}

\newcommand{\MYnewlabel}[1]{%
  \renewcommand\MYcurrentlabel{#1}%
  \MYoldlabel{#1}%
}

\newcommand{\MYdummylabel}[1]{}

\newcommand{\torestate}[1]{%
  \let\MYoldlabel\label%
  \let\label\MYnewlabel%
  #1%
  \MYstore{\MYcurrentlabel}{#1}%
  \let\label\MYoldlabel%
}

\newcommand{\restatetheorem}[1]{%
  \let\MYoldlabel\label
  \let\label\MYdummylabel
  \begin{theorem*}[Restatement of \prettyref{#1}]
    \MYload{#1}
  \end{theorem*}
  \let\label\MYoldlabel
}

\newcommand{\restatelemma}[1]{%
  \let\MYoldlabel\label
  \let\label\MYdummylabel
  \begin{lemma*}[Restatement of \prettyref{#1}]
    \MYload{#1}
  \end{lemma*}
  \let\label\MYoldlabel
}

\newcommand{\restateprop}[1]{%
  \let\MYoldlabel\label
  \let\label\MYdummylabel
  \begin{proposition*}[Restatement of \prettyref{#1}]
    \MYload{#1}
  \end{proposition*}
  \let\label\MYoldlabel
}

\newcommand{\restatefact}[1]{%
  \let\MYoldlabel\label
  \let\label\MYdummylabel
  \begin{fact*}[Restatement of \prettyref{#1}]
    \MYload{#1}
  \end{fact*}
  \let\label\MYoldlabel
}

\newcommand{\restate}[1]{%
  \let\MYoldlabel\label
  \let\label\MYdummylabel
  \MYload{#1}
  \let\label\MYoldlabel
}

\newcommand{\addreferencesection}{
  \phantomsection
\ifnum\stocmode=0
  \addcontentsline{toc}{section}{References}
\else
  \addcontentsline{toc}{section}{References \hspace*{1in} --------- End of extended abstract ---------}
\fi

}

\newcommand{\e}{\epsilon}

\renewcommand{\paragraph}[1]{\medskip\noindent{\bf #1.}}

\allowdisplaybreaks

\usepackage{paralist}

\usepackage{comment}

\let\pref=\prettyref

\newcommand{\diam}{\mathrm{diam}}

\DeclareMathOperator{\tr}{tr}

\newcommand{\vertiii}[1]{{\left\vert\kern-0.25ex\left\vert\kern-0.25ex\left\vert #1 
          \right\vert\kern-0.25ex\right\vert\kern-0.25ex\right\vert}}

\newcommand\f{\varphi}

\usepackage{accents}

\usepackage{stmaryrd}

\usepackage{enumitem}

\renewcommand{\mathbb}{\vvmathbb}

\begin{document}

\title{Spectral hypergraph sparsification via chaining}
\author{James R. Lee\thanks{Computer Science \& Engineering, University of Washington. \tt{jrl@cs.washington.edu}}}
\date{}

\maketitle

\begin{abstract}
   In a hypergraph on $n$ vertices where $D$ is the maximum size of a hyperedge,
   there is a weighted hypergraph spectral $\e$-sparsifier
   with at most $O(\e^{-2} \log(D) \cdot n \log n)$ hyperedges.
   This improves over the bound of Kapralov, Krauthgamer, Tardos and Yoshida (2021)
   who achieve $O(\e^{-4} n (\log n)^3)$, as well as the bound $O(\e^{-2} D^3 n \log n)$
   obtained by Bansal, Svensson, and Trevisan (2019).
   The same sparsification result was obtained independently by Jambulapati, Liu, and Sidford (2022).
\end{abstract}

\begingroup
\hypersetup{linktocpage=false}
\setcounter{tocdepth}{2}
\tableofcontents
\endgroup

\section{Introduction}

Consider a weighted hypergraph $H=(V,E,w)$ with $w \in \R_+^E$ and the corresponding energy:
For $x \in \R^V$,
\[
   Q_H(x) \seteq \sum_{e \in E} w_e \max_{\{u,v\} \in {e \choose 2}} (x_u-x_v)^{2}
\]
The problem of minimizing the energy $Q_H$ over various convex bodies
occurs in many applied contexts, especially in machine learning; we refer to the discussion in \cite{KKTY21b}.

In the graph case---when all the hyperedges have cardinality $2$---this corresponds to the quadratic form associated to the
weighted Laplacian and carries a physical interpretation as the potential energy of a family of springs
indexed by $\{u,v\} \in E$ whose respective endpoints are pinned at $x_u$ and $x_v$.
Let us mention the appealing analog for hypergraphs: If we stretch a rubber band around vertices
pinned at locations $\{ x_u : u \in e \}$, then $\max_{\{u,v\} \in {e \choose 2}} (x_u-x_v)^2$
is proportional to its potential energy. Here the weight $w_e$ represents the elasticity of the band.

\smallskip

For hypergraphs, the edge set $E$ could have cardinality as large $2^{|V|}$, and one can ask if there is a substantially
smaller hypergraph that approximates the energy for every configuration of vertices.
Soma and Yoshida \cite{SY19} formalized the following notion of spectral sparsification for hypergraphs, 
generalizing the well-studied notion for graphs \cite{ST11}.
Say that a weighted hypergraph $\tilde{H}=(V,\tilde{E},\tilde{w})$ is a {\em spectral $\e$-sparsifier} for $H$ if
$\tilde{E} \subseteq E$, and
\begin{equation}\label{eq:eps-sparsifier}
   |Q_H(x) - Q_{\tilde{H}}(x)| \leq \e Q_H(x),\qquad \forall x \in \R^V\,.
\end{equation}

We will use $n \seteq |V|$ throughout.
The authors \cite{SY19} showed that one can always find a spectral $\e$-sparsifier $\tilde{H}$ with
$|\tilde{E}| \leq O(n^3/\e^2)$.
In \cite{BST19}, the authors established a bound of $O(\e^{-2} D^3 n \log n)$, where $D \seteq \max \{ |e| : e \in E \}$
is often called the {\em rank of $H$,} and subsequently the authors of \cite{KKTY21a} achieved an upper bound of $n D (\e^{-1} \log n)^{O(1)}$.

\smallskip

Finally, in a recent and remarkable breakthrough, the authors of \cite{KKTY21b} show that one can obtain a spectral sparsifier with at most
$O(n (\log n)^3/\e^4)$ hyperedges, bypassing the polynomial dependence on the rank, and coming within $\poly(\e^{-1} \log n)$ factors of
the optimal bound.
By refining their approach via Talagrand's powerful generic chaining theory, we obtain the following improvement.

\begin{theorem}\label{thm:main}
   For any $n$-vertex weighted hypergraph $H=(V,E,w)$ and $\e > 0$, there is a spectral $\e$-sparsifier
   $\tilde{H}=(V,\tilde{E},\tilde{w})$ for $H$ with
   \[
      |\tilde{E}| \leq O\left(\frac{\log D}{\e^2} n \log n\right)\,,
   \]
   where $D \seteq \max_{e \in E} |e|$.
\end{theorem}

As in many prior works,
\pref{thm:main} is proved by defining a distribution on $E$
and then sampling edges independently from this distribution.
For approaches based on independent sampling, the bound of \pref{thm:main} is tight up to a constant factor
for every fixed $D$. In particular, this generalizes the analysis of independent
random sampling for graph sparsifiers \cite{SS11} where $D=2$.

It should be noted that for {\em cut sparsifiers}, the $\log D$ factor can be removed \cite{CKN20}.
This corresponds to the weaker notion where we only require that \eqref{eq:eps-sparsifier} holds
for $x \in \{-1,1\}^V$.
Whether the $\log D$ factor can be removed in general remains an intriguing open question.

Our proof of \pref{thm:main} entails an algorithm for constructing the sparsifier $\tilde{H}$ whose
running time is polynomial in the size of the input. But our sampling analysis can also be applied
directly to the faster algorithm presented in \cite{KKTY21b} whose running time is $|E|D \poly(\log |E|)+ \poly(n)$.

\medskip

\pref{thm:main} was proved independently and concurrently
by Jambulapati, Liu, and Sidford \cite{JLS22}, via a closely related approach.
While their main chaining result is somewhat less general than the one proved here (see \eqref{eq:lem11prev} below),
they also present a near-linear time algorithm for generating suitable sampling probabilities $\{\mu_e : e \in E \}$.
This improves the running time to $|E| D \poly(\log |E|)$.

\subsection{The random selector method and chaining for subgaussian processes}

Suppose we have a probability distribution $\mu \in \R_+^E$ on hyperedges in $H$.
We sample hyperedges $\tilde{E} = \{e_1,e_2,\ldots,e_M\}$ independently according to $\mu$,
and define the random weighted hypergraph $\tilde{H}=(V,\tilde{E},\tilde{w})$ so that
\[
   Q_{\tilde{H}}(x) = \frac{1}{M} \sum_{k=1}^M \frac{w_{e_k}}{\mu_{e_k}} Q_{e_k}(x)\,,
\]
where we define
\[
   Q_e(x) \seteq \max_{\{i,j\} \in {e \choose 2}} (x_i-x_j)^2\,,
\]
and the edge weights
\begin{equation}\label{eq:edge-weights}
   \tilde{w}_e \seteq \frac{\# \left\{ k \in [M] : e_k = e \right\}}{M} \cdot \frac{w_e}{\mu_e}\,.
\end{equation}
In particular, this gives $\E[Q_{\tilde{H}}(x)] = Q_{H}(x)$ for all $x \in \R^V$.

\smallskip

Now in order to find a spectral $\e$-sparsifier, we want to choose $M$ sufficiently large so that
\[
   \E \max_{x : Q_H(x) \leq 1} \left|Q_H(x)-Q_{\tilde{H}}(x)\right| \leq \e\,.
\]

To control concentration of $Q_{\tilde{H}}(x)$ around its mean, it suffices
to bound the average maximal fluctuations.
Thus by a standard sort of reduction (see \pref{sec:sampling} and also \cite[Lem 9.1.11]{TalagrandBook2014} for a general
formulation), it suffices to prove that for any
{\em fixed} hyperedges $e_1,\ldots,e_M \in E$,
\begin{equation}\label{eq:suffices}
   \E \max_{x : Q_H(x) \leq 1} \sum_{k=1}^M \e_k \frac{w_{e_k}}{\mu_{e_k}} Q_{e_k}(x) \leq O(\e M)\,,
\end{equation}
where $\e_1,\ldots,\e_M \in \{-1,1\}$ are i.i.d. random signs.

Thus our task is now to control the left-hand side of \eqref{eq:suffices}.
If we define the random variable 
\[
   V_x \seteq \sum_{k=1}^M \e_k \frac{w_{e_k}}{\mu_{e_k}} Q_{e_k}(x)\,,
\]
then $\{V_x : x \in \R^n \}$ is a subgaussian process (defined in \eqref{eq:subgaussian-tail}) with respect to the (semi)metric
\[
   d(x,\hat{x}) \seteq \left(\sum_{k=1}^M \left(\frac{w_{e_k}}{\mu_{e_k}}\right)^2 \left|Q_{e_k}(x)-Q_{e_k}(\hat{x})\right|^2\right)^{1/2}.
\]
There are well-developed tools for studying quantities like $\E \max \{ V_x : Q_H(x) \leq 1 \}$, but they rely on an understanding
of the geometry of the space $(\R^n,d)$, and a correct choice of distribution $\mu$ is essential for making
this geometry well-behaved.

\paragraph{Importance sampling}
For spectral graph sparsification, one chooses the sampling probability $\mu_e$ to be proportional
to the effective resistance across $e$ \cite{SS11}.
In order to extend this to hypergraphs, the authors of \cite{BST19}
define sampling probabilities $\{\mu_e : e \in E\}$
derived from the graph $G=(V,F)$, where $F \seteq \bigcup_{e \in E} {e \choose 2}$ is a union
of cliques on every hyperedge. They
take
\[
   \mu_e \propto \sum_{\{u,v\} \in {e \choose 2}} \sfR_{uv}\,, 
\]
where $\sfR_{uv}$ denotes the effective resistance
between a pair of vertices $u,v$ in $G$.

To remove the polynomial dependence on $D$, the authors of \cite{KKTY21b} choose
a {\em weighted graph} $G=(V,F,c)$ and define
\[
   \mu_e \propto w_e  \max \left\{ \sfR_{uv} : \{u,v\} \in \textstyle{{e \choose 2}} \right\}.
\]
Now $\sfR_{uv}$ is the effective resistance in $G$, where edges $\{u,v\} \in F$
have conductance $c_{uv}$.

\smallskip

Let $L_G$ denote the corresponding (weighted) graph Laplacian, and use
$L_G^{+}$ to denote its pseudoinverse.
Define $T \seteq \{ v \in \R^n : Q_H(L_G^{+/2} v) \leq 1\}$.
This construction of the sampling probabilities allows us to write
\begin{equation}\label{eq:s2}
   \E \max_{Q_H(x) \leq 1} V_x = \E \max_{v \in T} \sum_{k=1}^M \e_k \max_{\{i,j\} \in e_k} \langle v, y_{ij}^{e_k}\rangle^2\,,
\end{equation}
for a family of vectors $\{ y_{ij}^{e_k} \}$ that depends on our choice of edge conductances $c \in \R_+^F$ in $G$.

A central component of this approach is the existence of conductances that ensure two key properties:
\begin{enumerate}
   \item $T \subseteq B_2^n \seteq \{ x \in \R^n : \|x\| \leq 1 \}$,
   \item $\|y_{ij}^{e_k}\| \leq O(\sqrt{n})$ for all $k=1,\ldots,M$ and $\{i,j\} \in {e_k \choose 2}$.
\end{enumerate}
We return to a discussion of these properties in a moment.

\paragraph{Chaining bounds}
Note that the right-hand side of \eqref{eq:s2} can be written as
\[
   \E \max_{v \in T} \sum_{k=1}^M \e_k N_k(v)^2,
\]
where $N_k$ is an $\ell_{\infty}$ norm on a subset of the coordinates of $Av$, and $A$ is a
matrix whose rows are the vectors $\{y_{ij}^{e_k}\}$.
Thus in \pref{sec:bernoulli}, we apply aspects of the generic chaining theory (see the extensive reference \cite{TalagrandBook2014})
to the analysis of such expected maxima.

\smallskip

For readers familiar with the theory, let us note that a bound of $|\tilde{E}| \leq O(\e^{-2} n (\log n)^3)$ 
in \pref{thm:main} follows from applying Dudley's entropy bound (cf. \eqref{eq:dudley}) in a straightforward way.
A bound of $|\tilde{E}| \leq O(\e^{-2} n (\log n)^2)$ follows from a deeper inequality
of Talagrand (see \pref{thm:talagrand-uc} and \pref{sec:warmup}) that exploits property (1) above,
that $T$ is a subset of the Euclidean unit ball.

\smallskip

Finally, in order to achieve $|\tilde{E}| \leq O(\e^{-2} \log(D) \cdot n \log n)$, we need to exploit
further structure of the norms $\{N_k\}$ in a novel way. Our approach is modeled after Rudelson's
geometric argument \cite{Rudelson99} which, roughly speaking,
handles the case where each $N_k$ is a $1$-dimensional norm,
as well as Talagrand's method of chaining via growth functionals
(see \pref{sec:growth-functionals} and \pref{sec:advanced}).

\smallskip

To state this bound, let us consider arbitrary norms $N_1,\ldots,N_M$ on $\R^n$.
Define:
\begin{align*}
   \kappa &\seteq \E \max_{k \in [M]} N_k(g)\,, \\
   \lambda &\seteq \max_{k \in [M]} \left(\E[N_k(g)^2]\right)^{1/2}\,,
\end{align*}
where $g$ is a standard $n$-dimensional Gaussian.
In \pref{sec:advanced}, we prove that for any $T \subseteq B_2^n$,
\begin{equation}\label{eq:lem11prev}
      \E \sup_{x \in T} \sum_{k=1}^M \e_k N_k(x)^2 \leq O\!\left(\lambda \sqrt{\log n} + \kappa\right)
      \cdot \sup_{x \in T} \left(\sum_{k=1}^M N_k(x)^2 \right)^{1/2}
\end{equation}

When $M=m$, each $N_k$ is a $1$-dimensional norm $N_k(x)\seteq |\langle x, a_k\rangle|$ for some $a_k \in \R^n$, and $T=B_2^n$, this lemma recovers Rudelson's
concentration bound for Bernoulli sums of rank-$1$ matrices \cite{Rudelson99b}
(as mentioned there, the inequality we state next is a
consequence of the noncommutative Khintchine inequalities \cite{LPP91}).

Observe that
$N_k(x)^2 = \langle x, a_k\rangle^2 = \langle x, a_k a_k^* x\rangle$, and using $\|\cdot\|_{op}$ to
denote the operator norm, the preceding bound asserts that
\[
   \E \left\|\sum_{k=1}^m \e_k a_k a_k^*\right\|_{op} = \E \max_{x \in B_2^n} \left\langle x, \left(\sum_{k=1}^m \e_k a_k a_k^*\right) x \right\rangle
   \leq O(\slog{(m+n)}) \max_{k \in [m]} \|a_k\| \cdot \left\|\sum_{k=1}^m a_k a_k^*\right\|_{op}^{1/2},
\]
where we use $\lambda \leq O(1) \max_{k \in [m]} \|a_k\|$ and $\kappa \leq O(\slog{m}) \max_{k \in [m]} \|a_k\|$.

When applying \eqref{eq:lem11prev} to hypergraph sparsification, one picks up an additional $\sqrt{\log D}$ factor because each $N_k$
is an $\ell_{\infty}$ norm on a subset of at most $D$ coordinates.

\begin{remark}
As far as we know, it is an open problem to replicate consequences of the noncommutative Khintchine bound for higher-rank matrices using chaining, i.e.,
in the setting where $N_k(x) = \|A_k x\|$ for matrices $A_1,\ldots,A_M$.
\end{remark}

\paragraph{Choosing good conductances}
In order to satisfy properties (1) and (2) above, 
one chooses nonnegative numbers
\[
   \left\{ c_{ij}^e \geq 0 : \{i,j\} \in \textstyle{{e \choose 2}}, e \in E \right\}
\]
for which
\begin{equation}\label{eq:cap-split}
   \sum_{\{i,j\} \in {e \choose 2}} c_{ij}^e = w_e,\qquad \forall e \in E\,.
\end{equation}
Define the edge conductances $c_{ij} \seteq \sum_{e \in E : \{i,j\} \in {e \choose 2}} c_{ij}^e$.
As argued in \pref{sec:choose-con}, any such choice satisfies property (1).

Let $\sfR_{ij}$ denote the effective resistance between $\{i,j\} \in F$ in the weighted graph $G=(V,F,c)$.
To satisfy property (2), it suffices that for all hyperedges $e \in E$,
the effective resistances $\sfR_{ij}$ are the same for all pairs $\{i,j\} \in {e \choose 2}$ with $c_{ij}^e > 0$.
(This continues to hold even if the resistances are only comparable up to universal constant factors.)

Let $J$ denote the all-ones matrix and consider
maximizing the quantity
\[
   \log \det(L_G+J)
\]
over all choices of $(c_{ij}^e)$ satisfying \eqref{eq:cap-split}.
This quantity is a concave function of the conductances $(c_{ij}^e)$ and
the KKT conditions for the maximizer establish the desired property for the effective resistances.
See \pref{sec:balanced}.

This is essentially a reformulation and simplification
of the method used in \cite{KKTY21b} for establishing the existence of nice conductances $c : F \to \R_+$.
It is also reminiscent of Barthe's method for analyzing the Gaussian maximizers of the Brascamp-Lieb (and reverse Brascamp-Lieb) inequalities \cite{Barthe98}
(see also the treatment in \cite{HM13}).

\subsection{Notation}

For two expressions $A$ and $B$, we will use the equivalent notations $A \lesssim B$ and $A \leq O(B)$
to denote that there is a constant
$C > 0$ such that $A \leq C B$.  If $A$ and $B$ depend on some parameters $\alpha_1,\alpha_2,\ldots$, we use the
notation $A \lesssim_{\alpha_1,\alpha_2,\ldots} B$ to denote that there is a number $C = C(\alpha_1,\alpha_2,\ldots)$
such that $A \leq C B$. We use $A \asymp B$ to denote the conjunction of $A \lesssim B$ and $B \lesssim A$.

A number of vector and matrix norms will appear in what follows. When $x \in \R^n$ is a vector, $\|x\|$ will
always refer to the standard Euclidean norm of $x$.
For a positive integer $M \geq 1$, we will sometimes use the notation $[M] \seteq \{1,2,\ldots,M\}$.

\section{Extrema of random processes}

\label{sec:bernoulli}

\subsection{Background on generic chaining}

A space $(T,d)$ is called a {\em $K$-quasimetric} if satisfies
\begin{enumerate}
   \item $d(x,y)=d(y,x)$ for all $x,y \in T$\,.
   \item $d(x,x) = 0$ for all $x \in T$\,.
   \item There is a constant $K > 0$ such that
      \begin{equation*}\label{eq:quasi-metric}
         d(x,y) \leq K \left(d(x,z)+d(z,y)\right),\qquad \forall x,y,z \in T\,.
      \end{equation*}
\end{enumerate}
Say that $(T,d)$ is a {\em quasimetric space} if $(T,d)$ is a $K$-quasimetric for some $K > 0$.

Consider a distance $d$ on $T$.
A random process $\{V_x : x \in T\}$ is said to be {\em subgaussian with respect to $d$}
if there is a number $\alpha > 0$ such that
\begin{equation}\label{eq:subgaussian-tail}
   \Pr\left(|V_x-V_y| > t\right) \leq \exp\left(-\alpha \frac{t^2}{d(x,y)^2}\right),\qquad t > 0\,.
\end{equation}

\paragraph{The generic chaining functional}
For a quasimetric space $(T,d)$, let us recall Talagrand's generic chaining functional \cite[Def. 2.2.19]{TalagrandBook2014}.
Define $N_h \seteq 2^{2^h}$. 
Then
\begin{equation}\label{eq:gamma2}
   \gamma_2(T,d) \seteq \inf_{\{\cA_h\}} \sup_{x \in T} \sum_{h=0}^{\infty} 2^{h/2} \diam_d(\cA_h(x))\,,
\end{equation}
where the infimum runs over all sequences $\{\cA_h : h \geq 0\}$ of partitions of $T$ satisfying $|\cA_h| \leq N_h$
for each $h \geq 0$.
Note that we use the notation $\cA_h(x)$ for the unique set of $\cA_h$ that contains $x$,
and $\diam_d(S) \seteq \sup_{x,y \in S} d(x,y)$ for $S \subseteq T$.
The next theorem constitutes the generic chaining upper bound; see \cite[Thm 2.2.18]{TalagrandBook2014}.

\begin{theorem}
   If $\{V_x : x \in T\}$ is a centered subgaussian process satisfying \eqref{eq:subgaussian-tail}
   with respect to a $K$-quasimetric $(T,d)$, then
\begin{equation}\label{eq:chaining}
   \E \sup_{x \in T} V_x \lesssim_{K,\alpha} \gamma_2(T,d)\,.
\end{equation}
\end{theorem}

Define the entropy numbers $e_h(T,d) \seteq \inf \{ \sup_{t \in T} d(t,T_h) : T_h \subseteq T, |T_h| \leq 2^{2^h} \}$.
This is the infimum of numbers $r > 0$ such that $T$ can be covered by at most $2^{2^h}$ balls of radius $r$.
A classical way of controlling $\gamma_2(T,d)$ is given by Dudley's entropy bound (see, e.g., \cite[Prop 2.2.10]{TalagrandBook2014}):
\begin{equation}\label{eq:dudley}
   \gamma_2(T,d) \lesssim \sum_{h \geq 0} 2^{h/2} e_h(T,d)\,.
\end{equation}

But often additional structure of the space $(T,d)$ allows one to improve on \eqref{eq:dudley}.
The next lemma is a consequence of \cite[Thm 4.1.11 \& (4.23)]{TalagrandBook2014}.
It actually holds whenever $T$ is the unit ball of a uniformly $2$-convex Banach space and $d$ is induced
by some (possibly different) norm.

\begin{theorem}\label{thm:talagrand-uc}
   Suppose that $T=B_2^n$ is the unit Euclidean ball in $\R^n$ and $\|\cdot\|_X$ is a norm on $\R^n$.
   Then,
   \[
      \gamma_2(T,\|\cdot\|_X) \lesssim \left(\sum_{h \geq 0} \left(2^{h/2} e_h(T, \|\cdot\|_X)\right)^2\right)^{1/2}.
   \]
\end{theorem}

In order to bound the entropy numbers $e_h(B_2^n, \|\cdot\|_X)$, we will use the following
classical fact; see, e.g., \cite[(3.15)]{LedouxTalagrand2011}.

\begin{lemma}[Dual Sudakov inequality]
   \label{lem:dual-sudakov}
   Let $B_2^n$ denote the unit Euclidean ball, and suppose that $\|\cdot\|_X$ is a norm on $\R^n$.
   Then
   \[
      e_h(B_2^n,\|\cdot\|_X) \lesssim 2^{-h/2} \E \|g\|_X,
   \]
   where $g$ is a standard $n$-dimensional Gaussian.
\end{lemma}

\begin{corollary}\label{cor:talagrand}
   Suppose $\|\cdot\|_X$ is a norm on $\R^n$, and furthermore
   that $\|\cdot\|_X \leq L \|\cdot\|$ for some $L \geq 1$. Then,
   \[
      \gamma_2(B_2^n,\|\cdot\|_X) \lesssim L + \sqrt{\log n} \E \|g\|_X\,,
   \]
   where $g$ is a standard $n$-dimensional Gaussian.
\end{corollary}

\begin{proof}
   A straightforward volume argument shows that any set of $\delta$-separated points in $(B_2^n, \|\cdot\|)$
   must have cardinality at most $(4/\delta)^n$, and therefore
   \[
      e_h(T,\|\cdot\|) \leq 4 \cdot N_h^{-1/n} = 4\cdot 2^{-2^h/n}\,.
   \]
   By assumption, we have $e_h(B_2^n,\|\cdot\|_X) \leq L \cdot e_h(B_2^n, \|\cdot\|)$,
   and therefore
   \[
      e_h(B_2^n,\|\cdot\|_X) \leq 4L \cdot (2^{-2^h/n}).
   \]

   Denote $S \seteq \sup_{h \geq 0} 2^{h/2} e_h(T,\|\cdot\|_X)$.
   Applying \pref{thm:talagrand-uc} yields, for any $h_0 \geq 0$,
   \begin{align*}
      \gamma_2(T,d) \lesssim S \sqrt{h_0} + 4L \left(\sum_{h \geq h_0} (2^{h/2} 2^{-2^h/n})^2\right)^{1/2}\,.
   \end{align*}
   Choosing $h_0 \geq 2 \log n$ bounds the latter sum by $O(1)$, yielding
   \[
      \gamma_2(T,d) \lesssim S \sqrt{\log n} + L\,.
   \]
   To conclude, use \pref{lem:dual-sudakov} to bound $S$.
\end{proof}

\subsection{Warm up}
\label{sec:warmup}

The next lemma will allow us to establish the existence of
hypergraph spectral sparsifiers with at most $O(\e^{-2} n (\log n)^2)$ hyperedges.
It also provides a nice warm up for the more delicate arguments in \pref{sec:advanced}.

Let $A : \R^n \to \R^m$ denote a linear operator.
We use the notation
\[
   \|A\|_{2\to \infty} \seteq \max_{\|x\| \leq 1} \|Ax\|_{\infty}\,.
\]
This is equal to the maximum $\ell_2$ norm of a row of $A$.
Define the norm
\[
   \|x\|_{A} \seteq \|Ax\|_{\infty}\,,
\] 
and let us observe the following.

\begin{lemma}\label{lem:cov1}
   If $g$ is a standard $n$-dimensional Gaussian, it holds that
   \[
      \E \|g\|_A \lesssim \|A\|_{2\to \infty} \sqrt{\log m}\,.
   \]
   In particular, \pref{lem:dual-sudakov} gives
   \[
      e_h(B_2^n, \|\cdot\|_A) \lesssim 2^{-h/2} \sqrt{\log m} \|A\|_{2\to \infty}\,.
   \]
\end{lemma}

\begin{proof}
   If $a_1,\ldots,a_m$ are the rows of $A$ and $g$ is an $n$-dimensional Gaussian, then
   \[
      \E \|Ag\|_{\infty} = \E \max_{i \in [m]} |\langle g,a_i\rangle| \lesssim \max_{i \in [m]} \|a_i\| \sqrt{\log m}
      = \|A\|_{2 \to \infty} \sqrt{\log m}\,.\qedhere
   \]
\end{proof}

Additionally, let $\f_1,\f_2,\ldots,\f_M : \R^m \to \R$ be arbitrary functions.

\begin{lemma}\label{lem:Esup-warm}
   For any subset $T \subseteq B_2^n$,
   it holds that
   \[
      \E \sup_{x \in T} \sum_{j=1}^M \e_j \f_j(A x)^2 \lesssim \sqrt{\log m \log n}  \left\|A\right\|_{2 \to \infty} \cdot \sup_{\substack{j \in [M], \\ \|z-z'\|_{\infty} \leq 1}} |\f_j(z)-\f_j(z')|
      \cdot \sup_{x \in T} \left(\sum_{j=1}^M \f_j(Ax)^2\right)^{1/2}\,,
   \]
   where $\e_1,\ldots,\e_M$ are i.i.d. Bernoulli $\pm 1$ random variables.
\end{lemma}

\begin{proof}
   Define
\begin{align}
   \alpha &\seteq \max_{j \in [M]} \sup_{\|z-z'\|_{\infty} \leq 1} |\f_j(z)-\f_j(z')| \,,\label{eq:w1} \\
   \beta &\seteq \sup_{x \in T} \left(\sum_{j=1}^M \f_j(Ax)^2\right)^{1/2},\label{eq:w2} \\
   V_x &\seteq \sum_{j=1}^M \e_j \f_j(Ax)^2\,,\nonumber
\end{align}
and note that $\{V_x : x \in \R^n \}$ is a subgaussian process with respect to the distance
   \[
      d(x,\hat{x}) \seteq \left(\sum_{j=1}^M \left|\f_j(Ax)^2 - \f_j(A\hat{x})^2\right|^2\right)^{1/2}\,.
   \]
   Thus in light of \eqref{eq:chaining}, it suffices to prove that
   \begin{equation}\label{eq:goal-w}
      \gamma_2(T,d) \lesssim \sqrt{\log m \log n} \|A\|_{2\to\infty} \cdot \alpha\beta\,.
   \end{equation}

Note that for $x,\hat{x} \in T$,
\begin{align}
   d(x,\hat{x})^2 
   &= \sum_{j=1}^M \left(\f_j(Ax) - \f_j(A\hat{x})\right)^2
         \left(\f_j(Ax) + \f_j(A\hat{x})\right)^2 \nonumber \\
   &\leq \,2 \sum_{j=1}^M \left(\f_j(Ax) - \f_j(A\hat{x})\right)^2
         \left(\f_j(Ax)^2 + \f_j(A\hat{x})^2\right) \nonumber \\
   &\stackrel{\mathclap{\eqref{eq:w1}}}{\leq} \,2 \alpha^2\,\|A(x-\hat{x})\|_{\infty}^2\sum_{j=1}^M 
         \left(\f_j(Ax)^2 + \f_j(A\hat{x})^2\right) \nonumber \\
   &\stackrel{\mathclap{\eqref{eq:w2}}}{\leq} \,4 \alpha^2\beta^2 \, \|x-\hat{x}\|^2_A\,. \label{eq:last-warm-line}
\end{align} 
In particular, we have
\begin{equation}\label{eq:aineq-1}
   \gamma_2(T, d) \leq 2\alpha\beta \cdot \gamma_2(T, \|\cdot\|_A) \leq 2 \alpha\beta \cdot\gamma_2(B_2^n, \|\cdot\|_A),
\end{equation}
where the last inequality uses $T \subseteq B_2^n$.

   We can thus apply \pref{lem:cov1} and \pref{cor:talagrand} with $\|\cdot\|_X = \|\cdot\|_A$ and $L \seteq \|A\|_{2\to \infty}$ to conclude that
   \[
      \gamma_2(B_2^n, \|\cdot\|_A) \lesssim \|A\|_{2\to \infty} \sqrt{\log m\log n}\,.
   \]
   Combining this with \eqref{eq:aineq-1} completes our verification of \eqref{eq:goal-w}.
\end{proof}

In \pref{sec:advanced}, we will obtain an improved bound by using convexity
in a stronger way. In particular, we will assume that each of the 
functions $\f_j$ in \pref{lem:Esup-warm} is a norm on $\R^m$.

\subsection{Growth functionals}
\label{sec:growth-functionals}

Talagrand introduced a powerful way to control $\gamma_2(T,d)$ via the existence
of certain growth functionals.
For $x \in T$ and $\rho > 0$, define the ball
\begin{equation}\label{eq:balldef}
   B_d(x,\rho) \seteq \{ y \in T : d(x,y) \leq \rho \}\,.
\end{equation}

\begin{definition}[Separated sets]\label{def:ar-separated}
   Let $(T,d)$ denote a metric space and consider numbers $a > 0, r \geq 4$.
   Say that subsets {\em $H_1,\ldots,H_m \subseteq T$ are $(a,r)$-separated} if
   \[
      H_{\ell} \subseteq B_d(x_{\ell}, a/r), \quad \ell=1,\ldots,m\,,
   \]
   where $x_1,\ldots,x_m \in T$ are points satisfying
   \begin{equation}\label{eq:sep-centers}
         a \leq d(x_{\ell},x_{\ell'}) \leq ar,\quad \forall \ell \neq \ell'.
   \end{equation}
\end{definition}

\begin{definition}[The growth condition]\label{def:growth}
   Consider nonnegative functionals $\{F_h : h \geq 0\}$ defined on subsets of a metric space $(T,d)$ and
   which satisfy the following two conditions for every $h \geq 0$:
   \begin{align*}
      F_h(S) &\leq F_h(S'),\qquad \forall S \subseteq S' \subseteq T\,, \\
      F_{h+1}(S) &\leq F_{h}(S),\qquad \forall \ S \subseteq T\,.
   \end{align*}
   Say that such functionals satisfy the {\em growth condition with parameters
   $r \geq 4$ and $c^* > 0$} if for any integer $h \geq 0$ and $a > 0$, the following
   holds true with $m=N_{h+1}$:
   For each collection of subsets $H_1,\ldots,H_m \subseteq T$ that are $(a,r)$-separated,
   we have
   \begin{equation}\label{eq:no-packing}
      F_{h}\left(\bigcup_{\ell \leq m} H_{\ell}\right) \geq c^* a 2^{h/2} + \min_{\ell \leq m} F_{h+1}(H_{\ell})\,.
   \end{equation}
\end{definition}

\begin{theorem}[{\cite[Thm 2.3.16]{TalagrandBook2014}}]
   \label{thm:f-chaining}
   Let $(T,d)$ be a $K$-quasimetric space and consider a sequence of functionals $\{F_h\}$ 
   satisfying the growth condition (cf. \pref{def:growth}) with parameters $r \geq 4$ and $c^* > 0$.
   Then,
   \[
      \gamma_2(T,d) \lesssim_K \frac{r}{c^*} F_0(T) + r\cdot\diam_d(T)\,.
   \]
\end{theorem}

\begin{remark}[Packing/covering duality]
For the reader encountering \pref{def:growth} and \pref{thm:f-chaining} for the first time,
the role of the functionals $\{F_h\}$ might appear mysterious.
Some intuition can be gained by considering the duality between covering and packing:
A set $S$ in some metric space can be covered by $m$ balls of radius $r > 0$ if it is impossible
to find $m$ points in $S$ that are pairwise separated by distance $r$.

The quantity $\gamma_2(T,d)$ (cf. \eqref{eq:gamma2}) is a sort of multiscale covering functional.
The growth functionals $\{F_h\}$ measure the ``size'' of packings of various cardinalities,
and \eqref{eq:no-packing} asserts a form of packing impossibility.
This makes \pref{thm:f-chaining} a multiscale analog of the simple packing/covering argument
recalled above.

Those familiar with convex optimization and duality may find the approach of \cite{BDOM21}
instructive in this regard. It is shown that the corresponding {\em fractional}
multiscale covering and packing values are equal by convex duality, and then a rounding argument establishes
that the integral versions are equivalent up to constant factors.
\end{remark}

We will use the following corollary of \pref{thm:f-chaining} that simplifies
the construction of functionals if we have a bound on the growth rate of nets in $(T,d)$.

\begin{corollary}\label{cor:f-chaining}
   Let $(T,d)$ be a $K$-quasimetric and
   assume there are numbers $k, L \geq 1$ and $r \geq 4$ such that
   that for every $a > 0$,
   \begin{equation}\label{eq:growth-cond}
      H_1,\ldots,H_m \subseteq T \textrm{ are $(a,r)$-separated} \implies
      m \leq \left(\frac{L}{a}\right)^k.
   \end{equation}
   Let $h_0$ be the largest integer $h \geq 0$ such that
   \begin{equation}\label{eq:hlk}
      2^{2^{h}} \leq 2^{k(h-1)/2}\,.
   \end{equation}
   Consider a sequence of functionals $\{F_0,F_1,\ldots,F_{h_0}\}$
   satisfying the growth condition \eqref{eq:no-packing} with parameters $r$ and $c^* > 0$.
   Then,
   \begin{equation}\label{eq:f-chaining}
      \gamma_2(T,d) \lesssim_K \frac{r}{c^*} F_0(T) + r \left(\diam_d(T) + L\right)\,.
   \end{equation}
\end{corollary}

\begin{proof}
   Define the numbers
   \begin{align*}
   c_j &\seteq c^* L \cdot 2^{-2^j/k} 2^{(j-1)/2} \\
   C_0 &\seteq \sum_{j=h_0+1}^{\infty} c_j\,,
   \end{align*}
   and note that $C_0 \lesssim c^* L$, since \eqref{eq:hlk} is violated for every $h \geq h_0 + 1$.

   Define a new family of functionals $\{\tilde{F}_h : h \geq 0\}$ so that for every $S \subseteq T$,
   \begin{align*}
      \tilde{F}_h(S) &\seteq F_h(S) + C_0\,, \quad && h=0,1,\ldots,h_0\,, \\
      \rightspace \tilde{F}_h(S) &\seteq F_{h_0}(S) + C_0 - \sum_{j=h_0+1}^{h} c_j\,, \quad && h > h _0\,. \rightspace\rightspace
   \end{align*}
   By construction, these satisfy the growth condition \pref{def:growth} since for $h \geq h_0$,
   if $H_1,\ldots,H_m \subseteq T$ are $(a,r)$-separated sets with $m=2^{2^{h+1}}$, then
   \[
      \tilde{F}_{h+1}\left(\bigcup_{\ell \leq m} H_{\ell}\right)
      \geq c_{h+1} + \tilde{F}_h\left(\bigcup_{\ell \leq m} H_{\ell}\right)
      \geq c_{h+1} + \min_{\ell \leq m} \tilde{F}_h\left(H_{\ell}\right)
      \geq c^* a 2^{h/2} + \min_{\ell \leq m} \tilde{F}_h\left(H_{\ell}\right),
   \]
   where the last inequality uses the fact that $a \leq L 2^{-2^{h+1}/k}$ from \eqref{eq:growth-cond}.
   Moreover, we have
   \[
      \tilde{F}_0(T) = F_0(T) + C_0 \leq F_0(T) + O(c^* L)\,,
   \]and therefore we can apply
   \pref{thm:f-chaining} to $\{\tilde{F}_h\}$ to complete the proof.
\end{proof}

\subsection{Further exploiting convexity}
\label{sec:advanced}

We will now use the growth functional approach (cf. \pref{sec:growth-functionals})
to prove a more elaborate upper bound under the additional assumption that our summands are derived from norms.
This will allow us in \pref{sec:h-sparse} to find spectral $\e$-sparsifiers with $O\!\left(\frac{\log D}{\e^2} n \log n\right)$ hyperedges.

\smallskip

Let $N_1,N_2,\ldots,N_M$ be norms on $\R^n$ and
define
\begin{align*}
   \kappa &\seteq \E \max_{j \in [M]} N_j(g)\,, \\
   \lambda &\seteq \max_{j \in [M]} \left(\E[N_j(g)^2]\right)^{1/2}\,,
\end{align*}
where $g$ is a standard $n$-dimensional Gaussian.

\begin{lemma}\label{lem:Esup}
   For any $T \subseteq B_2^n$, it holds that
   \[
      \E \sup_{x \in T} \sum_{j=1}^M \e_j N_j(x)^2 \lesssim \left(\lambda \sqrt{\log n} + \kappa\right)
      \cdot \sup_{x \in T} \left(\sum_{j=1}^M N_j(x)^2 \right)^{1/2}\,,
   \]
   where $\e_1,\ldots,\e_M$ are i.i.d. Bernoulli $\pm 1$ random variables.
\end{lemma}

Before proving the lemma, let us illustrate a corollary that we will use 
to construct hypergraph sparsifiers.
Consider a linear operator $A : \R^n \to \R^m$, and
suppose that each $N_i$ is a (weighted) $\ell_{\infty}$ norm on some subset $S_i \subseteq [m]$
of the coordinates:
\begin{equation}\label{eq:norm-def}
   N_i(z) = \max_{j \in S_i} w_j |(Az)_j|\,,\qquad w \in [0,1]^{S_i}\,.
\end{equation}
Let $a_1,\ldots,a_m$ denote the rows of $A$, and observe that $(Ag)_j = \langle a_j,g\rangle$
is a normal random variable with variance $\|a_j\|^2$, and therefore
\[
   \E[N_i(g)]^2 = \max_{j \in S_i} w_j^2 |\langle a_j,g\rangle|^2 \lesssim \max_{j \in S_i} \|a_j\|^2 \cdot \sqrt{\log |S_i|}\,.
\]
Similarly, we have 
\[
   \kappa = \E \max_{i \in [M]} \max_{j \in S_i} w_j^2 |\langle a_j,g\rangle|^2 \leq \E \max_{i \in [m]} |\langle a_i,g\rangle|^2 \lesssim \|A\|_{2\to\infty} \sqrt{\log m}\,.
\]

\begin{corollary}\label{cor:lp-version}
   If the norms $N_1,\ldots,N_M$ are of the form \eqref{eq:norm-def} for some 
   $A : \R^n \to \R^m$ and
   subsets $S_1,\ldots,S_M \subseteq [m]$ with $\max_{i \in [M]} |S_i| \leq D$, then
   for any $T \subseteq B_2^n$, it holds that
   \[
      \E \sup_{x \in T} \sum_{j=1}^M \e_j N_j(x)^2 \lesssim \|A\|_{2\to\infty} \sqrt{\log (m+n) \log D}
      \cdot \sup_{x \in T} \left(\sum_{j=1}^M N_j(x)^2 \right)^{1/2}\,,
   \]
   where $\e_1,\ldots,\e_M$ are i.i.d. Bernoulli $\pm 1$ random variables.
\end{corollary}

The proof of \pref{lem:Esup} is modeled after arguments
of Rudelson \cite{Rudelson99} and Talagrand; see \cite[\S 16.7]{TalagrandBook2014} and
the historical notes in \cite[\S 16.10]{TalagrandBook2014}.
A version of the latter argument first appeared in \cite{Rudelson99}, as a simplification of Rudelson's original construction
of an explicit majorizing measure.
In the proof of \cite[Prop 16.7.4]{TalagrandBook2014},
one encounters growth functionals of the form $F(S) = 1 - \inf \{ \|u\| : u \in \conv(S) \}$, where $\|\cdot\|$
is a uniformly $2$-convex norm. We recall this definition.

\begin{definition}[Uniform $p$-convexity]
   A Banach space $Z$ is called {\em uniformly $p$-convex}
   if there is a number $\eta > 0$ such that for all $x,y \in Z$ with $\|x\|_Z,\|y\|_Z \leq 1$,
   \[
      \left\|\frac{x+y}{2}\right\|_Z \leq 1 - \eta \|x-y\|_Z^p\,.
   \]
\end{definition}

\smallskip

We remark that the statement of \pref{lem:Esup} actually holds when $T$ is a subset of the unit ball of
any uniformly $2$-convex norm on $\R^n$ (with an implicit constant that depends on $\eta$).

We will instead employ functionals of the form 
\[
   F(S) = 2 - \inf \left\{ \|u\|^2 + \sum_{j=1}^M N_j(u)^2 : u \in \conv(S) \right\}.
\]
Problematically, the norm $u \mapsto \left(\|u\|^2 + \sum_{j=1}^M N_j(u)^2\right)^{1/2}$ is potentially very far from uniformly
$2$-convex,
thus we have to be careful in using only $2$-convexity of the Euclidean norm,
along with $2$-convexity of the ``outer'' $\ell_2$ norm of the $N_j$'s.
This requires application of the inequality $|N_j(x)-N_j(\hat{x})| \leq N_j(x-\hat{x})$
only at judiciously chosen points in the argument.
We offer some further explanation in \pref{rem:discussion} after the proof.

\begin{proof}[Proof of \pref{lem:Esup}]
   For a set $S \subseteq \R^n$, let $\conv(S)$ denote the closed convex hull of $S$.
   Note that by convexity,
   \[
      \sup_{x \in T} \left(\sum_{j=1}^M N_j(x)^2\right)^{1/2} = \sup_{x \in \conv(T)} \left(\sum_{j=1}^M N_j(x)^2\right)^{1/2}.
   \]
   Therefore we may replace $T$ by $\conv(T)$ and henceforth assume that $T$ is compact and convex.

   \smallskip

   By scaling $\{N_j\}$, we may assume that
   \begin{align}
      \sup_{x \in T} \sum_{j=1}^M N_j(x)^2 &= 1\,. \label{eq:opnorm1}
   \end{align}

Define $V_x \seteq \sum_{j=1}^M \e_j N_j(x)^2$.
Then $\{V_x : x \in \R^n\}$ is a subgaussian process with respect to the metric
\[
   \tilde{d}(x,\hat{x}) \seteq \left(\sum_{j=1}^M |N_j(x)^2 - N_j(\hat{x})^2|^2\right)^{1/2},
\]
therefore from \eqref{eq:chaining}, we have
\begin{equation}\label{eq:chaining0}
   \E \sup_{x \in T} V_x \lesssim \gamma_2(T,\tilde{d})\,.
\end{equation}

\paragraph{Passing to a nicer distance}
Define the related distance
\begin{align*}
d(x,\hat{x}) &\seteq \left(\sum_{j=1}^M N_j(x-\hat{x})^2 \left(N_j(x)^2 + N_j(\hat{x})^2\right)\right)^{1/2},
\end{align*}
and note that for all $x,\hat{x} \in \R^n$,
\begin{align*}
   \tilde{d}(x,\hat{x})^2 &= \sum_{j=1}^M \left(N_j(x)-N_j(\hat{x})\right)^2 \left(N_j(x) + N_j(\hat{x})\right)^2 \nonumber \\
                          &\leq 2 \sum_{j=1}^M N_j(x-\hat{x})^2 \left(N_j(x)^2 + N_j(\hat{x})^2\right) = 2\, d(x,\hat{x})^2\,. \label{eq:basic-compare}
\end{align*}

We will observe momentarily that
\begin{equation}\label{eq:real-qm}
   d(x,\hat{x}) \leq 2\sqrt{2} \left(d(x,y)+d(y,\hat{x})\right),\qquad \forall x,\hat{x},y \in \R^n\,.
\end{equation}
Since $\tilde{d} \leq \sqrt{2} d$ and $d$ is a quasimetric, \eqref{eq:chaining} gives
\begin{equation*}\label{eq:chaining1}
   \E \sup_{x \in T} V_x \lesssim \gamma_2(T,d)\,,
\end{equation*}
and thus our goal is to establish that
\begin{equation}\label{eq:goal}
   \gamma_2(T,d) \lesssim \lambda \sqrt{\log n} + \kappa\,.
\end{equation}

\begin{lemma}\label{lem:simple-qm}
   For any metric space $(X,D)$ and $x_0 \in X$, it holds that
   the distance
   \[
      \tilde{D}(x,\hat{x}) \seteq D(x,\hat{x}) \left(D(x,x_0)+D(\hat{x},x_0)\right)
   \]
   is a $2$-quasimetric.
\end{lemma}

\begin{proof}
   Define $\psi(x) \seteq D(x,x_0)$ and
   consider $x,\hat{x},y \in X$.
   Then,
   \begin{align*}
      \tilde{D}(x,\hat{x}) &\leq (D(x,y) + D(\hat{x},y)) \left(\psi(x)+\psi(\hat{x})\right) \\
                   &\leq D(x,y) \left(\psi(x) + \psi(y) + D(\hat{x},y)\right) + D(\hat{x},y) \left(\psi(\hat{x}) + \psi(y) + D(x,y)\right) \\
                   &\leq \tilde{D}(x,y) + \tilde{D}(\hat{x},y) + 2 D(x,y) D(\hat{x},y)\,.
   \end{align*}
   Now use $2 D(x,y) D(\hat{x},y) \leq D(x,y)^2 + D(\hat{x},y)^2 \leq \tilde{D}(x,y) + \tilde{D}(\hat{x},y)$, completing the proof.
\end{proof}

Applying the preceding lemma
with $D(x,\hat{x}) = N_j(x-\hat{x})$ and $x_0=0$
shows that the distance $(x,\hat{x}) \mapsto N_j(x-\hat{x}) (N_j(x)+N_j(\hat{x})^2)^{1/2}$ is a $2\sqrt{2}$-quasimetric for each $j=1,\ldots,M$,
and therefore $d$ is a $2\sqrt{2}$-quasimetric on $\R^n$, verifying \eqref{eq:real-qm}.

\paragraph{Balls in $(\R^n,d)$ are approximately convex}
Recall the definition of the balls $B_d(x,\rho)$ from \eqref{eq:balldef}.

\begin{lemma}\label{lem:approx-convex}
   For any $x \in \R^n$ and $\rho > 0$, it holds that
   \[
      \conv(B_d(x,\rho)) \subseteq B_d(x,4\rho)\,.
   \]
\end{lemma}

\begin{proof}
   For $y \in B_d(x,\rho)$, we have
   \begin{equation}\label{eq:bnd1}
      \left(\sum_{j=1}^M N_j(x-y)^2 N_j(x)^2\right)^{1/2} \leq \rho\,,
   \end{equation}
   as well as
   \begin{equation}
      \sqrt{\rho} \geq d(x,y)^{1/2} = \left(\sum_{j=1}^M N_j(x-y)^2 \left(N_j(x)^2+N_j(y)^2\right)\right)^{1/4} \geq \left(\frac12 \sum_{j=1}^M N_j(x-y)^4\right)^{1/4},\label{eq:bnd2}
   \end{equation}
   where the final inequality uses $N_j(x-y) \leq N_j(x)+N_j(y)$.
   Since the left-hand side of \eqref{eq:bnd1} and the right-hand side of \eqref{eq:bnd2} are both convex functions of $y$, these inequalities remain true
   for all $y \in \conv(B_d(x,\rho))$.

   In particular, for any $y \in \conv(B_d(x,\rho))$, we can use $a^2+b^2 \leq 4a^2 + 2(a-b)^2$ to write
   \begin{align*}
      d(x,y) &\leq \left(\sum_{j=1}^M N_j(x-y)^2 \left(4 N_j(x)^2 + 2 (N_j(x)-N_j(y))^2\right)\right)^{1/2} \\
             &\leq 2 \left(\sum_{j=1}^M N_j(x-y)^2 N_j(x)^2\right)^{1/2} + \sqrt{2} \left(\sum_{j=1}^M N_j(x-y)^4\right)^{1/2} \leq 4 \rho\,.\qedhere
   \end{align*}
\end{proof}

\paragraph{Covering estimates}
Define now the following norms on $\R^n$:
\begin{align*}
   \|x\|_{\cN} &\seteq \max_{j \in [M]} N_j(x)\,, \\
   \|x\|_{\cE(u)} &\seteq \left(\sum_{j=1}^M N_j(x)^2 N_j(u)^2\right)^{1/2}, \quad u \in \R^n\,.
\end{align*}

\begin{lemma}\label{lem:dist-norms}
   For all $x,\hat{x},u \in \R^n$,
   \begin{align*}
      d(x,\hat{x})^2 &\leq
      2\, \|x-\hat{x}\|_{\cN}^2 \left(\sum_{j=1}^M \left(N_j(x)-N_j(u)\right)^2+\sum_{j=1}^M \left(N_j(\hat{x})-N_j(u)\right)^2\right) + 4 \|x-\hat{x}\|_{\cE(u)}^2\,.
   \end{align*}
\end{lemma}

\begin{proof}
   Use the inequalities
   \begin{align}
      N_j(x)^2 \leq 2 (N_j(x)-N_j(u))^2 + 2N_j(u)^2\,, \qquad x,u \in \R^n\nonumber
   \end{align}
   to write
   \begin{align*}
      \sum_{j=1}^M N_j(x-\hat{x})^2 N_j(x)^2
                                    & \leq 2 \|x-\hat{x}\|_{\cN}^2 \sum_{j=1}^M \left(N_j(x)-N_j(u)\right)^2 +
      2 \sum_{j=1}^M N_j(x-\hat{x})^2 N_j(u)^2 \\
                                    &=  2 \|x-\hat{x}\|_{\cN}^2 \sum_{j=1}^M \left(N_j(x)-N_j(u)\right)^2 +
      2 \|x-\hat{x}\|^2_{\cE(u)}\,.\qedhere
   \end{align*}
\end{proof}

\begin{lemma}\label{lem:cov2}
   It holds that 
   \begin{align*}
      e_h(B_2^n, \|\cdot\|_{\cN}) &\lesssim 2^{-h/2} \kappa\,,\\
      e_h(B_2^n, \|\cdot\|_{\cE(u)}) &\lesssim 2^{-h/2} \lambda\,,\quad \forall u \in T\,.
   \end{align*}
\end{lemma}

\begin{proof}
   Both inequalities follow readily from \pref{lem:dual-sudakov}: If $g$ is a standard $n$-dimensional Gaussian, then
   \[
      e_h(B_2^n, \|\cdot\|_{\cN}) \lesssim 2^{-h/2} \E \|g\|_{\cN} = 2^{-h/2} \kappa,
   \]
   by the definition of $\kappa$.
   For the second inequality,
   \[
      e_h(B_2^n, \|\cdot\|_{\cE(u)}) \lesssim 2^{-h/2} \E \|g\|_{\cE(u)}\,.
   \]
   Now use convexity of the square to bound
   \[
      \left(\E \|g\|_{\cE(u)}\right)^2 \leq \E \|g\|_{\cE(u)}^2 = \sum_{j=1}^M N_j(u)^2 \E [N_j(g)^2]
      \leq \lambda^2\,,
   \]
   where the final line uses the definition of $\lambda$ and $\sum_{j=1}^M N_j(u)^2 \leq 1$ by \eqref{eq:opnorm1}, because $u \in T$.
\end{proof}

We also need a basic estimate that we will use to apply \pref{cor:f-chaining}.
Observe that for $x,\hat{x} \in T$,
\begin{equation}\label{eq:diam-bnd}
   d(x,\hat{x}) \stackrel{\substack{\eqref{eq:opnorm1}}}{\leq} \sqrt{2} \|x-\hat{x}\|_{\cN}
   \leq \sqrt{2} \left(\|x\|_{\cN} + \|\hat{x}\|_{\cN}\right) \leq 2 \sqrt{2}\,,
\end{equation}
where the last inequality uses $\|x\|_{\cN} \leq (\sum_{j=1}^M N_j(x)^2)^{1/2} \leq 1$ for $x \in T$, by \eqref{eq:opnorm1}.

\begin{lemma}
   \label{lem:simple-count}
   For any $a > 0$,
   if $x_1,\ldots,x_K \in T$ satisfy $d(x_i,x_j) \geq a$ for $i \neq j$,
   then, $K \leq \left(\frac{6}{a}\right)^n$.
\end{lemma}

\begin{proof}
   As noted above, we have $\|x\|_{\cN} \leq 1$ for $x \in T$, and \eqref{eq:diam-bnd} gives 
   $\|x_i-x_j\|_{\cN} \geq a/\sqrt{2}$ for $i \neq j$.
   Therefore by a simple volume argument (valid for any norm on $\R^n$):
   \[
      K \leq \left(1+\frac{2\sqrt{2}}{a}\right)^n \leq \left(\frac{6}{a}\right)^n,
   \]
   where the last inequality follows because if $K \geq 2$, then \eqref{eq:diam-bnd} implies $a \leq 2\sqrt{2}$.
\end{proof}

\paragraph{The growth functionals}
Define a norm on $\R^n$ by
\begin{equation}\label{eq:iii}
   \vertiii{u} \seteq \left(\|u\|^2 + \sum_{j=1}^M N_j(u)^2\right)^{1/2}\,.
\end{equation}
Denote $r \seteq 64$.
Let $h_0$ be the largest integer so that $2^{2^{h_0}} \leq 2^{n(h-1)/2}$, and
note that $h_0 \leq O(\log n)$.
Define
\begin{align}
   F_h(S) &\seteq 2 - \inf \left\{ \vertiii{u}^2 : u \in \conv(S) \right\} + \frac{\max(h_0+1-h,0)}{\log n}, \qquad && h=0,1,\ldots,h_0\,.\label{eq:Fh}
\end{align}
Recall that $T \subseteq B_2^n$ and, along with  \eqref{eq:opnorm1}, this gives $\max_{u \in T} \vertiii{u}^2 \leq 2$.
Since $h_0 \leq O(\log n)$, we have $F_0(T) \leq O(1)$.

\smallskip

From \eqref{eq:diam-bnd}, we have $\diam_d(T) \leq O(1)$.
Note also that from \pref{lem:simple-count}, it holds that
the packing assumption \eqref{eq:growth-cond} is satisfied with $L \leq O(1)$
and $k = n$.
Therefore if we can verify that our functionals satisfy the growth conditions \eqref{eq:no-packing}
for $h=0,1,\ldots,h_0$, then we will conclude from \eqref{eq:f-chaining} that
\begin{equation}\label{eq:gamma-pre}
   \gamma_2(T,d) \lesssim \frac{1}{c^*} + 1\,.
\end{equation}

\medskip

\paragraph{Consideration of $(a,r)$-separated sets}
Define $K \seteq N_{h+1}$ and consider points $\{x_{1},\ldots,x_K \} \subseteq T$
such that $d(x_{\ell},x_{\ell'}) \geq a$ whenever $\ell \neq \ell'$, along with sets
$H_{\ell} \subseteq T \cap B_d(x_{\ell}, a/r)$ for $\ell=1,\ldots,K$.

\medskip

Let $z_0$ be a minimizer of $\vertiii{u}^2$ over $u \in\conv(\bigcup_{\ell \leq K} H_{\ell})$,
and note that $z_0 \in T$ since $T$ is closed and convex.
Define $\theta_0 \seteq \vertiii{z_0}^2$ and
\[
   \theta \seteq \max_{\ell \leq K} \min \left\{ \vertiii{u}^2 : u \in \conv(H_{\ell}) \right\},
\]
and for each $\ell \in [K]$, let $z_{\ell} \in \conv(H_{\ell})$ be such that $\vertiii{z_{\ell}}^2 \leq \theta$.

Note that $\conv(H_{\ell}) \subseteq \conv(B_d(x_{\ell},a/r)) \subseteq B_d(x_{\ell},4a/r)$, where the latter inclusion
follows from \pref{lem:approx-convex}.
Since $z_{\ell} \in \conv(H_{\ell})$, we have $d(x_{\ell},z_{\ell}) \leq 4a/r$ for all $\ell \in \{1,\ldots,K\}$.
In particular for $\ell,\ell' \in \{1,\ldots,K\}$ with $\ell \neq \ell'$, we can use the quasimetric inequalities \eqref{eq:real-qm} to write
\begin{align*}
   a \leq d(x_{\ell},x_{\ell'}) &\leq 2 \sqrt{2} \left(d(x_{\ell},z_{\ell}) + d(z_{\ell},x_{\ell'})\right) \\
                         &\leq 2 \sqrt{2}\,\frac{4a}{r} + 8 \left(d(z_{\ell},z_{\ell'}) + d(z_{\ell'},x_{\ell'})\right)
                         \leq (8+2\sqrt{2}) \frac{4a}{r} + 8\, d(z_{\ell},z_{\ell'}).
\end{align*}
Using our choice $r = 64$, we conclude that
that for $\ell \neq \ell'$,
\begin{equation}\label{eq:still-sep}
   d(z_{\ell},z_{\ell'}) \geq \frac{a}{32}\,.
\end{equation}

Observe that
\[
   F_h\left(\bigcup_{\ell \leq m} H_{\ell}\right) - \min_{\ell \leq K} F_{h+1}(H_{\ell}) = (2-\theta_0) - (2-\theta) + \frac{1}{\log n}
   = \theta - \theta_0 + \frac{1}{\log n}\,,
\]
thus to verify that the growth condition \pref{def:growth} holds for $\{F_h\}$,
our goal is to show that
\begin{equation}\label{eq:goal-growth}
   \theta - \theta_0 + \frac{1}{\log n} \gtrsim \frac{2^{h/2} a}{\kappa + \lambda \sqrt{\log n}}\,,\qquad h=0,1,\ldots,h_0\,.
\end{equation}
This will confirm the growth condition with $c^* \asymp \left(\lambda \sqrt{\log n} + \kappa\right)^{-1}$,
and therefore \eqref{eq:gamma-pre} yields our desired goal \eqref{eq:goal}.

\smallskip

The next lemma exploits $2$-uniform convexity of the $\ell_2$ distance.
Note that the claimed inequality would fail (in general) if the left-hand side were replaced by the
larger quantity $\vertiii{z_0-z_{\ell}}^2$,
as $\vertiii{\cdot}$ is not necessarily $2$-convex.

\begin{lemma}\label{lem:uc2}
   For every $\ell=1,\ldots,K$, it holds that
   \[
      \left\|z_0-z_{\ell}\right\|^2 + \sum_{j=1}^M \left(N_j(z_0)-N_j(z_{\ell})\right)^2 \leq 2 (\theta-\theta_0)\,.
   \]
\end{lemma}

\begin{proof}
   Let us use
   \[
      \left(\frac{a-b}{2}\right)^2 = \frac12 a^2 + \frac12 b^2 - \left(\frac{a+b}{2}\right)^2.
   \]
   to write
\begin{align*}
   \left\|\frac{z_0-z_{\ell}}{2}\right\|^2 + \sum_{j=1}^M \left(\frac{N_j(z_0)-N_j(z_{\ell})}{2}\right)^2
                                  &= \frac12 \left(\|z_{\ell}\|^2 + \sum_{j=1}^M N_j(z_{\ell})^2\right)
                                  + \frac12 \left(\|z_{0}\|^2 + \sum_{j=1}^M N_j(z_0)^2\right) \\
                                  & \qquad\qquad - \left\|\frac{z_0+z_{\ell}}{2}\right\|^2 - \sum_{j=1}^M \left(\frac{N_j(z_0)+N_j(z_{\ell})}{2}\right)^2.
\end{align*}
By convexity of the norm $N_j$, we have $\frac12 (N_j(z_0)+N_j(z_{\ell})) \geq N_j(\tfrac{z_0+z_{\ell}}{2})$, so the preceding
identity gives
\begin{align*}
   \left\|\frac{z_0-z_{\ell}}{2}\right\|^2 + \sum_{j=1}^M \left(\frac{N_j(z_0)-N_j(z_{\ell})}{2}\right)^2
                                  &\leq \frac12 \vertiii{z_{\ell}}^2 + \frac12 \vertiii{z_0}^2 - \vertiii{\tfrac{z_0+z_{\ell}}{2}}^2 \\
                                  &\leq \vertiii{z_{\ell}}^2 - \vertiii{\tfrac{z_0+z_{\ell}}{2}}^2 \\
                                  &\leq \theta - \theta_0\,,
\end{align*}
where the inequality $\vertiii{\frac{z_0+z_{\ell}}{2}}^2 \geq \theta_0$
follows from $\frac{z_0+z_{\ell}}{2} \in \conv(\bigcup_{\ell \leq K} H_{\ell})$,
since $z_0 \in \conv(\bigcup_{\ell \leq K} H_{\ell})$ and $z_{\ell} \in \conv(H_{\ell})$.
\end{proof}

Define $\rho \seteq \theta - \theta_0$. 
One consequence of \pref{lem:uc2} is that 
\[
   z_1,\ldots,z_K \in z_0 + \sqrt{2 \rho} B_2^n\,.
\]
We can cover $z_0 + \sqrt{2\rho} B_2^n$ by $N_h$ sets
that have $\|\cdot\|_{\cN}$-diameter bounded by $2 e_h(\sqrt{2\rho} B_2^n, \|\cdot\|_{\cN})$.
Since we have $K = N_{h+1} = N_h^2$ points $z_1,\ldots,z_K$, at least $N_{h}$
of them $z_{i_1},\ldots,z_{i_{N_h}}$ must lie in the same set of the cover.
And by definition, these points cannot all have pairwise $\|\cdot\|_{\cE(z_0)}$
distance greater than
$e_h(\sqrt{2\rho} B_2^n, \|\cdot\|_{\cE(z_0)})$.
Therefore we must have at least two points $z_{\ell}$ and $z_{\ell'}$ with $\ell \neq \ell'$ and $\ell,\ell' \geq 1$,
and such that
\begin{align*}
   \|z_{\ell}-z_{\ell'}\|_{\cN} &\leq 2 e_h(\sqrt{2\rho} B_2^n, \|\cdot\|_{\cN})\lesssim 2^{-h/2} \kappa \sqrt{\rho} \,, \\
   \|z_{\ell}-z_{\ell'}\|_{\cE(z_0)} &\leq e_h(\sqrt{2\rho} B_2^n, \|\cdot\|_{\cE(z_0)}) \lesssim 2^{-h/2} \lambda \sqrt{\rho}\,,
\end{align*}
where the latter two estimates follow from \pref{lem:cov1} and \pref{lem:cov2}, respectively.

Let us also note a second consequence of \pref{lem:uc2}, that 
\[
   \sum_{j=1}^M \left(N_j(z_0)-N_j(z_{\ell})\right)^2 + 
   \sum_{j=1}^M \left(N_j(z_0)-N_j(z_{\ell'})\right)^2 \leq 4 \rho\,.
\]
Using the three preceding inequalities in \pref{lem:dist-norms} yields
\[
   a^2 \stackrel{\eqref{eq:still-sep}}{\lesssim} d(z_{\ell},z_{\ell'})^2 \lesssim 2^{-h} \rho^2 \kappa^2 + 2^{-h} \rho \lambda^2
   \leq \max\left(2^{-h} \kappa^2 \rho^2, 2^{-h} \lambda^2 \rho\right).
\]

This implies
\[
   \rho \gtrsim \min\left(\frac{2^{h/2} a}{\kappa},\frac{2^h a^2}{\lambda^2}\right).
\]
Since it holds that
\[
   \frac{2^h a^2}{\lambda^2} + \frac{1}{\log n} \geq \frac{2^{h/2} a}{\lambda \sqrt{\log n}},
\]
we conclude that
\[
   \rho + \frac{1}{\log n} \gtrsim
   \min\left(\frac{2^{h/2} a}{\kappa},\frac{2^{h/2} a}{\lambda \sqrt{\log n}}\right)
   \gtrsim \frac{2^{h/2} a}{\lambda \sqrt{\log n} + \kappa}\,.
\]
Recalling that $\rho = \theta - \theta_0$, we have established \eqref{eq:goal-growth},
completing the proof.
\end{proof}

\begin{remark}[Discussion of the implicit partitioning]
\label{rem:discussion}
It is often more intuitive to think about bounding $\gamma_2(T,d)$ by explicitly
constructing the sequence of partitions $\{\cA_h\}$ (recall \eqref{eq:gamma2}).
This is a technical process that is aided significantly by \pref{thm:f-chaining}, whose proof involves
the construction of partitions from growth functionals.

Recall the norm $\vertiii{\cdot}$ from \eqref{eq:iii}
and for a subset $S \subseteq B_2^n$, define the quantity
\[
   \f(S) \seteq 2 - \min \left\{ \vertiii{x}^2 : x \in \conv(S) \right\}.
\]
Then $\f(S)$ can be considered as an approximate measure of the ``size'' of $S$,
where sets of larger $\f(S)$ value tend to have a larger $\E \sup_{x \in S} \sum_{j=1}^M \e_j N_j(x)^2$ value.

\smallskip

Recall that $r \seteq 64$.
Consider a ball $B_d(x_0,\eta)$, and let $z_0 \in B_d(x_0,4\eta)$ be such that $\f(B_d(x_0,\eta)) = 2-\vertiii{z_0}^2$.
Let us think of $z_0$ as the ``analytic center'' of the ball $B_d(x_0,\eta)$.
(We have to take $z_0 \in B_d(x_0,4\eta)$ because the ball $B_d(x_0,\eta)$ is only approximately convex.)

Define the distance
\[
   \Delta(x,y) \seteq \left(\|x-y\|^2 + \sum_{j=1}^M (N_j(x)-N_j(y))^2\right)^{1/2} \,,\qquad x,y \in \R^n\,.
\]
For $x \in B_d(x_0,\eta)$, let $\hat{x} \in B_d(x,4\eta/r^2)$ denote a point satisfying
$\f(B_d(x,\eta/r^2)) = 2 - \vertiii{\hat{x}}^2$.
Then \pref{lem:uc2} gives
\begin{equation}\label{eq:size-fact}
   \Delta(z_0,\hat{x})^2  \lesssim \f\left(B_d(x_0, \eta)\right) - \f\left(B_d(x, \eta/r^2)\right)\,.
\end{equation}
In other words, either the $\f$-value of $B_d(x, \eta/r^2)$ is significantly smaller than that of $B_d(x_0,\eta)$, or $\hat{x}$
is close (in the distance $\Delta$) to the analytic center $z_0$.

The second part of the argument involves bounding the number of centers that can be within a certain distance of $z_0$.
Consider now any points $x_1,\ldots,x_M \in B_d(x_0,\eta)$ with $d(x_i,x_j) > \eta/r$ for $i \neq j$.
\pref{lem:dist-norms} and the covering estimates on $e_h(B_2^n, \|\cdot\|_{\cE(z_0)})$ and $e_h(B_2^n, \|\cdot\|_{\cN})$ together
give that for some constant $C > 0$,
\begin{equation}\label{eq:card-fact}
\# \left\{ i \geq 1 : \Delta(z_0, \hat{x}_i)^2 \leq \rho \right\} \leq \exp\left(\frac{C}{\eta^2} \left(\kappa^2 \rho^2 + \lambda^2 \rho\right)\right).
\end{equation}
Now \eqref{eq:size-fact} and \eqref{eq:card-fact} imply that for any $\delta > 0$,
\begin{equation}\label{eq:key-tradeoff}
   \# \left\{ i \geq 1 : \f\left(B_d(x_i,\eta/r^2)\right) \geq \f\left(B_d(x_0,\eta)\right) - \delta \right\} \leq\exp\left(\frac{C}{\eta^2} \left(\kappa^2 \delta^2 + \lambda^2 \delta\right)\right).
\end{equation}
This is the key tradeoff occuring in the argument: A bound on the number of pairwise separated ``children''
$B_d(x_i,\eta/r^2)$ of $B_d(x_0,\eta)$ that do not experience a significant reduction in their $\f$-value.

Employing this bound repeatedly, in a sufficiently careful manner, allows one to construct a sequence of partitions $\{\cA_h\}$
that yields the desired upper bound on $\gamma_2(T,d)$. The role of \pref{thm:f-chaining} is to automate this process.
\end{remark}

\section{Hypergraph sparsification}
\label{sec:h-sparse}

Suppose $H=(V,E,w)$ is a weighted hypergraph and denote $n \seteq |V|$.
For a single hyperedge $e \in E$, let us recall the definitions
\[
   Q_e(x) \seteq \max_{\{u,v\} \in {e \choose 2}} (x_u-x_v)^2\,,
\]
as well as the energy
\[
   Q_H(x) \seteq \sum_{e \in E} w_e Q_e(x)\,.
\]

\subsection{Sampling}
\label{sec:sampling}

Suppose we have a probability distribution $\mu \in \R_+^E$ on hyperedges in $H$.
Let us sample hyperedges $\tilde{E} = \{ e_1,e_2,\ldots,e_M \}$ independently according to $\mu$.
The weighted hypergraph $\tilde{H}=(V,\tilde{E},\tilde{w})$ is defined so that
\[
   Q_{\tilde{H}}(x) = \frac{1}{M} \sum_{k=1}^M \frac{w_{e_k}}{\mu_{e_k}} Q_{e_k}(x)\,,
\]
In particular, $\E[Q_{\tilde{H}}(x)] = Q_{H}(x)$ for all $x \in \R^V$.
Recall that the hyperedge weights in $\tilde{H}$ are given by \eqref{eq:edge-weights}.
To help us choose the distribution $\mu$, we now introduce a Laplacian on an auxiliary graph.

\paragraph{An auxiliary Laplacian}
Define the edge set $F \seteq \bigcup_{e \in E} {e \choose 2}$,
and let
$G=(V,F,c)$ be a weighted graph, where we will choose the edge conductances $c \in \R_+^F$ later.
Denote by $L_G : \R^V \to \R^V$ the weighted Laplacian
\begin{equation}\label{eq:weighted-lap}
   L_G \seteq \sum_{\{i,j\} \in F} c_{ij} (\chi_i-\chi_j) (\chi_i-\chi_j)^*,
\end{equation}
where $\chi_1,\ldots,\chi_n$ is the standard basis of $\R^n$.
Let $L_G^{+}$ denote its Moore-Penrose pseudoinverse and define
\begin{align}
   \sfR_{ij} &\seteq \|L_G^{+/2} (\chi_i-\chi_j)\|^2,  
   \qquad\quad\, \{i,j\} \in F\,,  \nonumber \\
   \sfR_{\max}(e) &\seteq \max \left\{ \sfR_{ij} : \{i,j\} \in \textstyle{{e \choose 2}} \right\}, \,\qquad  e \in E\,, \nonumber \\
   Z &\seteq \sum_{e \in E} w_e \sfR_{\max}(e)\,, \nonumber \\
   \mu_e &\seteq \frac{w_e \sfR_{\max}(e)}{Z}\,, \label{eq:p-def} \quad\qquad\qquad\qquad 
   e\ \, \in E\,.
\end{align}

\begin{lemma}\label{lem:hsparse}
   Suppose it holds that
   \begin{equation}\label{eq:norm-compare-cond}
      \|x\|^2 \leq Q_H(L_G^{+/2} x)\,,\qquad  \forall x \in \R^n\,.
   \end{equation}
   Then for any $\e \in (0,1)$, there is a number 
   \[
      M_0 \lesssim \frac{\log D}{\e^2} Z \log n
   \]
   such that for $M \geq M_0$,
   with probability at least $1/2$, the hypergraph $\tilde{H}$ is a spectral $\e$-sparsifier for $H$.
\end{lemma}

\begin{proof}
By convexity,
\begin{equation}\label{eq:start}
   \E_{\tilde{H}} \max_{v : Q_H(v) \leq 1} \left|Q_H(v)-Q_{\tilde{H}}(v)\right| \leq
   \E_{\tilde{H},\hat{H}} \max_{v : Q_H(v) \leq 1} |Q_{\tilde{H}}(v)-Q_{\hat{H}}(v)|\,,
\end{equation}
where $\hat{H}$ is an independent copy of $\tilde{H}$.

The latter quantity can be written as
\begin{align}
   \nonumber
   \E_{\tilde{e},\hat{e}} &\max_{v : Q_H(v) \leq 1} \left|\frac{1}{M} \sum_{i=1}^M \frac{w_{\tilde{e}_i}}{\mu_{\tilde{e}_i}} Q_{\tilde{e}_i}(v) - \frac{1}{M} \sum_{i=1}^M \frac{w_{\hat{e}_i}}{\mu_{\hat{e}_i}} Q_{\hat{e}_i}(v)\right| \\
   &= 
   \E_{\e} \E_{\tilde{e},\hat{e}} \max_{v : Q_H(v) \leq 1} \left|\frac{1}{M} \sum_{i=1}^M \e_i \left(\frac{w_{\tilde{e}_i}}{\mu_{\tilde{e}_i}}Q_{\tilde{e}_i}(v) - \frac{w_{\hat{e}_i}}{\mu_{\hat{e}_i}} Q_{\hat{e}_i}(v)\right)\right| \label{eq:sign-trick} \\
   &\leq 2 \E_{\tilde{H}} \E_{\e} \max_{v : Q_H(v) \leq 1} \left|\frac{1}{M} \sum_{i=1}^M \e_i \frac{w_{e_i}}{\mu_{e_i}} Q_{e_i}(v)\right|,\label{eq:latterq}
\end{align}
where $\e_1,\ldots,\e_M$ are i.i.d. Bernoulli $\pm 1$ random variables.
Note that we can introduce signs in \eqref{eq:sign-trick} because the distribution of
$\frac{w_{\tilde{e}_i}}{\mu_{\tilde{e}_i}}Q_{\tilde{e}_i}(v) - \frac{w_{\hat{e}_i}}{\mu_{\hat{e}_i}} Q_{\hat{e}_i}(v)$
is symmetric.

For $e \in E$ and $\{i,j\} \in {e \choose 2}$, define the vectors
\begin{align*}
   y_{ij} &\seteq L_G^{+/2} (\chi_i-\chi_j) \\
   y^e_{ij} &\seteq \sqrt{\frac{w_e}{\mu_e}}\  y_{ij} = \sqrt{\frac{Z}{\sfR_{\max}(e)}}\ y_{ij}\,.
\end{align*}
Then we have
\begin{equation}\label{eq:edge-exp}
\frac{w_e}{\mu_e} Q_e(L_G^{+/2} x)  = \frac{w_e}{\mu_e} \max_{\{i,j\} \in {e \choose 2}}|\langle L_G^{+/2} x, \chi_i-\chi_j\rangle|^2
= \max_{\{i,j\} \in {e \choose 2}}\langle x,y_{ij}^e\rangle^2\,.
\end{equation}

Define the values
\[
   S_{ij} \seteq \max_{e \in E : \{i,j\} \in {e \choose 2}} \|y_{ij}^e\|\,,\quad \{i,j\} \in F,
\]
and the linear map $A : \R^n \to \R^F$ by $(Ax)_{\{i,j\}} \seteq S_{ij} \langle x,y_{ij}/\|y_{ij}\|\rangle$.

\smallskip

For $k=1,\ldots,M$, define the weighted $\ell_{\infty}$ norms
\[
   N_k(z) \seteq \max \left\{ \left|(Az)_{\{i,j\}}\right| \frac{\|y_{ij}^{e_k}\|}{S_{ij}} : \{i,j\} \in {e_k \choose 2}, S_{ij} > 0 \right\}.
\]
It holds that
\[
   N_k(x) = \max_{\{i,j\} \in e_k} |\langle x,y_{ij}^{e_k}\rangle|\,,
\]
so from \eqref{eq:edge-exp}, we have
\begin{align}
   Q_{\tilde{H}}(L_G^{+/2} x) &= \frac{1}{M} \sum_{i=1}^M N_i(x)^2\,, \label{eq:rewrite1} \\
   \frac{1}{M} \sum_{i=1}^M \e_i \frac{w_{e_i}}{\mu_{e_i}} Q_{e_i}(L_G^{+/2} x) &= 
   \frac{1}{M} \sum_{i=1}^M \e_i N_i(x)^2\,.\label{eq:rewrite0}
\end{align}
Thus we can write the quantity \eqref{eq:latterq} as
\[
   2\E_{\tilde{H}} \E_{\e} \max_{x : Q_H(L_G^{+/2} x) \leq 1} \left|\frac{1}{M} \sum_{i=1}^M \e_i N_i(x)^2\right|
   \leq 4 \E_{\tilde{H}} \E_{\e} \max_{x : Q_H(L_G^{+/2} x) \leq 1} \frac{1}{M} \sum_{i=1}^M \e_i N_i(x)^2,
\]

Define $T \seteq \{ x \in \R^n : Q_H(L_G^{+/2} x) \leq 1 \}$ and note that from \eqref{eq:norm-compare-cond}, we have $T \subseteq B_2^n$.
Now apply \pref{cor:lp-version} to bound
\begin{equation}\label{eq:ucbnd}
    \E_{\e} \max_{x \in T} \frac{1}{M} \sum_{i=1}^M \e_i N_i(x)^2 \lesssim
    \frac{\|A\|_{2\to\infty} \sqrt{\log n \log D}}{M^{1/2}}  \max_{x \in T} \left(\frac{1}{M} \sum_{i=1}^M  N_i(x)^2\right)^{1/2}.
\end{equation}
Note also that
\[
   \max_{x \in T} \frac{1}{M} \sum_{i=1}^M  N_i(x)^2 =
    \max_{v : Q_H(v) \leq 1} \frac{1}{M} \sum_{i=1}^M N_i\left(L_G^{1/2} v\right)^2=
   \max_{v : Q_H(v) \leq 1} Q_{\tilde{H}}(v)\,.
\]
where the first equality follows from the fact that $Q_H(x)=Q_H(\hat{x})$ when $x-\hat{x} \in \ker(L_G)$, and
the second inequality uses this and an application of \eqref{eq:rewrite1} with $x = L_G^{1/2} v$.

Recalling our starting point \eqref{eq:start}, it follows that for some universal constant $C > 0$,
\begin{align*}
   \tau \seteq \E_{\tilde{H}} \max_{v : Q_H(v) \leq 1} \left|Q_H(v)-Q_{\tilde{H}}(v)\right| &\leq
   C \frac{\|A\|_{2\to\infty} \sqrt{\log n \log D}}{M^{1/2}} 
   \E_{\tilde{H}} \left(\max_{v : Q_H(v) \leq 1} Q_{\tilde{H}}(v)\right)^{1/2} \\
                                                                                &\leq
                                                                                C \frac{\|A\|_{2\to\infty} \sqrt{\log n \log D}}{M^{1/2}} 
\left(\E_{\tilde{H}} \max_{v : Q_H(v) \leq 1} Q_{\tilde{H}}(v)\right)^{1/2},
\end{align*}
where the last inequality is by concavity of the square root.

Observe that
\[
   \max_{v : Q_H(v) \leq 1} Q_{\tilde{H}}(v) \leq \max_{v : Q_H(v) \leq 1} \left(\left|Q_H(v)-Q_{\tilde{H}}(v)\right|+Q_H(v)\right)
\leq 1 + \max_{v : Q_H(v) \leq 1} |Q_H(v) - Q_{\tilde{H}}(v)|\,,
\]
and therefore we have
\[
   \tau \leq C \frac{\|A\|_{2\to\infty} \sqrt{\log n \log D}}{M^{1/2}}  \left(1+\tau\right)^{1/2}\,.
\]
It follows that if $M \geq (2C\|A\|_{2\to\infty} \sqrt{\log n \log D})^2$, then $\tau \leq 4C  \frac{\|A\|_{2\to\infty} \sqrt{\log n \log D}}{M^{1/2}}$.

For $0 < \e < 1$, choosing 
\[
   M \seteq \frac{4C^2 \log D}{\e^2} \|A\|_{2\to\infty}^2 \log n
\]
gives
\[
   \E_{\tilde{H}} \max_{v : Q_H(v) \leq 1} \left|Q_H(v)-Q_{\tilde{H}}(v)\right| = \tau \leq \e\,.
\]

The proof is complete once we observe that
\[
   \|A\|^2_{2\to\infty} = \max_{\{i,j\} \in F} S_{ij}^2 = \max_{e \in E, \{i,j\} \in {e \choose 2}} \|y_{ij}^e\|^2 = Z \max_{\{i,j\} \in {e \choose 2}} \frac{\sfR_{ij}}{\sfR_{\max}(e)}
   \leq Z\,.
   \qedhere
\]
\end{proof}

\subsection{Choosing conductances}
\label{sec:choose-con}

We are therefore left to find edge conductances in the graph $G=(V,F,c)$ so that 
\eqref{eq:norm-compare-cond} holds and $Z$ is small.
To this end, let us choose
nonnegative numbers
\[
   \left\{ c_{ij}^e \geq 0 : \{i,j\} \in {e \choose 2}, e \in E \right\}
\]
such that
\begin{equation}\label{eq:cap1}
   \sum_{\{i,j\} \in {e \choose 2}} c_{ij}^e = w_e, \quad \forall e \in E\,.
\end{equation}

For $\{i,j\} \in F$, we then define our edge conductance
\begin{equation}\label{eq:he-split}
   c_{ij} \seteq \sum_{e \in E : \{i,j\} \in {e \choose 2}} c_{ij}^e\,.
\end{equation}
In this case,
\begin{align*}
   \|L_G^{1/2} v\|^2 = \langle v, L_G v\rangle &= \sum_{\{i,j\} \in F} c_{ij} (v_i-v_j)^2 \\
                                               &= \sum_{e \in E} 
                                               \sum_{\{i,j\} \in {e \choose 2}} c_{ij}^e (v_i-v_j)^2 \\
                                               &\leq \sum_{e \in E} 
                                               \sum_{\{i,j\} \in {e \choose 2}} c_{ij}^e \max_{\{i,j\} \in {e \choose 2}} (v_i-v_j)^2 \\
                                               &\stackrel{\mathclap{\eqref{eq:cap1}}}{\leq}\ \ 
                                               \sum_{e \in E} w_e \max_{\{i,j\} \in {e \choose 2}} (v_i-v_j)^2 = Q_H(v)\,.
\end{align*}
Taking $v = L_G^{+/2} x$ gives
\[
   \|x\|^2 \leq Q_H(L_G^{+/2} x),
\]
verifying \eqref{eq:norm-compare-cond}.

\begin{lemma}[Foster's Network Theorem]
   \label{lem:foster}
   It holds that $\sum_{\{i,j\} \in F} c_{ij} \sfR_{ij} \leq n-1$.
\end{lemma}

\begin{proof}
   Recall that $\sfR_{ij} = \langle \chi_i-\chi_j, L_G^+ (\chi_i-\chi_j)\rangle$ and $L_G = \sum_{\{i,j\} \in F} c_{ij} (\chi_i-\chi_j) (\chi_i-\chi_j)^*$.
   It follows that
   \[
      \sum_{\{i,j\} \in F} c_{ij} \sfR_{ij} = \sum_{\{i,j\} \in F} \tr(c_{ij}(\chi_i-\chi_j)(\chi_i-\chi_j)^* L_G^+) = \tr(L_G L_G^+) \leq n-1\,,
   \]
   since $\rank(L_G) \leq n-1$.
\end{proof}

Define
\begin{equation}\label{eq:hatk}
   K \seteq \max_{e \in E} \max_{\{i,j\} \in {e \choose 2}} \frac{\sfR_{\max}(e)}{\sfR_{ij}} \1_{\{c_{ij}^e > 0\}}
\end{equation}
so that
\[
   Z = \sum_{e \in E} w_e \sfR_{\max}(e) = \sum_{e \in E} \sum_{\{i,j\} \in {e \choose 2}} c_{ij}^e \sfR_{\max}(e)
   \leq K\sum_{e \in E} \sum_{\{i,j\} \in {e \choose 2}} c_{ij}^e \sfR_{ij} \leq K (n-1)\,,
\]
where the last inequality uses \eqref{eq:he-split} and \pref{lem:foster}.
In conjunction with \pref{lem:hsparse}, we have proved the following.

\begin{lemma}\label{lem:hsparse2}
   Suppose there is a choice of conductances so that \eqref{eq:cap1} holds.
   Then for any $\e > 0$, there is a spectral $\e$-sparsifier for $H$ with at most $O(K \frac{\log D}{\e^2} n \log n)$ hyperedges,
   where $K$ is defined in \eqref{eq:hatk}.
\end{lemma}

\subsection{Balanced effective resistances}
\label{sec:balanced}

We will exhibit conductances satisfying \eqref{eq:cap1} and \eqref{eq:hatk} with $K \leq 1$.
To this end, we may assume that the weighted hypergraph $H=(V,E,w)$ has strictly positive edge weights
and that the (unweighted) graph $G_0=(V,F)$ is connected.

Define $\hat{F} \seteq \{ (e,\{i,j\}) : e \in E, \{i,j\} \in {e \choose 2} \}$, and
consider vectors $\left(c_{ij}^e : e \in E, \{i,j\} \in {e \choose 2}\right) \in \R_+^{\hat F}$.
Define the convex set
\[
   \sfK \seteq \R_+^{\hat{F}} \cap
   \left\{ \sum_{\{i,j\} \in {e \choose 2}} c_{ij}^e = w_e : e \in E \right\}.
\]

We use $\cS_+^n$ and $\cS_{++}^n$ for the cones of positive semidefinite (resp., positive definite) $n \times n$ matrices.
Define $c_{ij} \seteq \sum_{e : \{i,j\} \in {e \choose 2}} c_{ij}^e$
and denote the linear function $L_G : \R_+^{F} \to \cS_+^n$ by
\[
   L_G\left((c_{ij})\right) \seteq \sum_{\{i,j\} \in F} c_{ij} (\chi_i-\chi_j)(\chi_i-\chi_j)^*\,.
\]
Let $J$ be the all-ones matrix and
consider the objective
\[
   \Phi\left((c_{ij})\right) \seteq - \log \det\left(L_G\left((c_{ij})\right) + J\right)\,.
\]

Note that $X \mapsto - \log \det(X)$ is a convex function on the cone $\cS_{+}^n$ of $n \times n$ positive semidefinite matrices (see, e.g., \cite[\S 3.1]{bv04})
and takes the value $+\infty$ on $\cS_{+}^n \setminus \cS_{++}^n$.
Consider finally the convex optimization problem:
\begin{equation}\label{eq:optimization}
   \min \left\{ \Phi\left((c_{ij})\right) : (c_{ij}^e) \in \sf{K} \right\}.
\end{equation}
Since $G_0$ is connected, it holds that if $\left(c_{ij}\right) \in \R_{++}^F$, then
$\ker(L_G)$ is the span of $(1,1,\ldots,1)$, and therefore $L_G\left((c_{ij})\right) + J \in \cS_{++}^n$.
Therefore $\Phi$ is finite on the strictly positive orthant $\R_{++}^F$.

\begin{lemma}\label{lem:ok-prog}
   The value of \eqref{eq:optimization} is finite and there is a feasible point in the relative interior of $\sfK$.
\end{lemma}

\begin{proof}
   It is straightforward to check that the maximum of eigenvalue of $L_G$ is bounded by
   $2 \sum_{\{i,j\} \in {e \choose 2}} c_{ij} = 2\sum_{e \in E} w_e$, hence the value of \eqref{eq:optimization} is finite.
   Moreover, the vector defined by $c_{ij}^e \seteq \frac{1}{|{e \choose 2}|} w_e$ is feasible
   and lies in $\R_{++}^{\hat{F}}$ since the weights $w_e$ are strictly positive.
\end{proof}

\smallskip

We can write the corresponding Lagrangian as
\begin{align*}
   g\left((c_{ij}^e); \alpha,\beta\right) = - \log &\det\left(L_G\left((c_{ij})\right) + J\right) 
   + \sum_{e \in E}  \alpha_e \left(\sum_{\{i,j\} \in {e \choose 2}} c_{ij}^e - w_e\right) - \sum_{e \in E} \sum_{\{i,j\} \in {e \choose 2}} \beta_{ij}^e c_{ij}^e
\end{align*}

\pref{lem:ok-prog} allows one to conclude that there are vectors $(\hat{c}_{ij}^e), \hat{\alpha},\hat{\beta}$ with $\hat{\beta} \geq 0$
and such that the KKT conditions hold;
see \cite[Thm 28.2]{Rockafellar70}.
In particular, for all $e \in E$ and $\{i,j\} \in {e\choose 2}$, we have
\begin{align}
   \partial_{c_{ij}^e}\, g\!\left((\hat{c}_{ij}^e); \hat{\alpha},\hat{\beta}\right) &= 0\,,\label{eq:kkt1} \\
   \hat{\beta}_{ij}^e > 0 \implies \hat{c}_{ij}^e &= 0\,. \label{eq:kkt2}
\end{align}

By the rank-one update formula for the determinant, we have
\[
   \partial_{c_{ij}^e} \log \det(L_G+ J) = \langle \chi_i-\chi_j, (L_G + J)^{-1} (\chi_i-\chi_j)\rangle\,.
\]
Define $\hat{L}_G \seteq L_G\left((\hat{c}_{ij})\right)$.
Define $\hat{\sfR}_{ij} \seteq \langle \chi_i-\chi_j, \hat{L}_G^{+} (\chi_i-\chi_j)\rangle$.
Taking the derivative of $g$ with respect to each $c_{ij}^e$ and using \eqref{eq:kkt1} gives
\[
   \hat{\sfR}_{ij} =
   \langle \chi_i-\chi_j, (\hat{L}_G+ J)^{-1} (\chi_i-\chi_j)\rangle = \hat{\alpha}_e - \hat{\beta}_{ij}^e,\qquad \forall e \in E, \{i,j\} \in {e \choose 2}\,,
\]
where the first equality uses the fact that the eigenvectors of $\hat{L}_G$ and $J$ are orthogonal and $\chi_i - \chi_j \in \ker{J}$.

Note that since $\hat{\beta} \geq 0$ coordinate-wise, this implies that
\[
   \hat{\sfR}_{\max}(e) \seteq \max_{\{i,j\} \in {e \choose 2}} \hat{\sfR}_{ij} \leq \hat{\alpha}_e\,.
\]
Moreover, if $\hat{c}_{ij}^e > 0$, then $\hat{\beta}^e_{ij} = 0$ (cf. \eqref{eq:kkt2}), and in that case $\hat{\sfR}_{ij} = \hat{\alpha}_e = \hat{\sfR}_{\max}(e)$.

\smallskip

We conclude that the edge conductances $\hat{c}_{ij}^e$ yield $K \leq 1$ in \eqref{eq:hatk}, and therefore
\pref{lem:hsparse2} gives a sparsifier with $O(\frac{\log D}{\e^2} n \log n)$ edges, completing the proof of \pref{thm:main}.

\subsection*{Acknowledgements}

I am grateful to Thomas Rothvoss for many suggestions and comments on preliminary drafts.

\bibliographystyle{alpha}
\bibliography{diffusive}

\end{document}